\documentclass[a4paper]{amsart}
\usepackage[all]{xy}

\usepackage{pstricks-add}
\usepackage{float}

\usepackage[colorlinks=true]{hyperref}
\usepackage{enumerate}
\usepackage{amssymb}
\usepackage{amsbsy}
\usepackage{comment}
\usepackage{aurical}
\usepackage{graphicx}
\usepackage{mathrsfs}
\usepackage[cal=boondox]{mathalfa}

\newtheorem{thm}{Theorem}[section]
\newtheorem{prop}[thm]{Proposition}
\newtheorem{lem}[thm]{Lemma}
\newtheorem{cor}[thm]{Corollary}

\theoremstyle{definition}
\newtheorem{defn}[thm]{Definition}

\theoremstyle{remark}
\newtheorem{rem}[thm]{Remark}

\newcommand{\Rr}{\mathbb R}

\def\bb#1#2{\left\{#1,#2\right\}}

\usepackage{color}
\usepackage{xcolor}
\usepackage{framed}

\definecolor{DarkBlue}{rgb}{0.1,0.1,0.5}
\definecolor{Red}{rgb}{0.9,0.0,0.1}
\definecolor{lightred}{rgb}{1, 0.75, 0.75}

\colorlet{light-yellow}{orange!60!yellow}
\colorlet{light-blue}{blue!15}
\colorlet{darkblue}{blue!70!black}

\begin{document}
\title{Split Courant algebroids as $L_{\infty}$-structures}
\author{P. Antunes}
\address{University of Coimbra, CMUC, Department of Mathematics, 3001-501 Coimbra, Portugal}
\email{pantunes@mat.uc.pt}
\author{J.M. Nunes da Costa}
\address{University of Coimbra, CMUC, Department of Mathematics, 3001-501 Coimbra, Portugal}
\email{jmcosta@mat.uc.pt}
\begin{abstract}
We show that split Courant algebroids, i.e., those defined on a Whitney sum $A \oplus A^*$, are in a one-to-one correspondence with multiplicative curved $L_\infty$-algebras. This one-to-one correspondence extends to Nijenhuis  morphisms and behaves well under the operation of twisting by a bivector.
\end{abstract}

%%% ----------------------------------------------------------------------
\maketitle
%%% ----------------------------------------------------------------------

\textbf{Mathematics Subject Classifications (2010).} Primary 53D17; Secondary 17B66, 17B70, 58A50.

\

\textbf{Keywords.} Courant algebroid, $L_\infty$-algebra, Nijenhuis morphism.

%%%%%%%%%%%%%%%%%%%%%%%%%%%%%%%%%%%%
%%%%%%%%%%%%%%%%%%%%%%%%%%%%%%%%%%%%
%%%%%%%%%%%%%%%%%%%%%%%%%%%%%%%%%%%%
%
\section{Introduction}             %
\label{section_introduction}           %
%%%%%%%%%%%%%%%%%%%%%%%%%%%%%%%%%%%%
%%%%%%%%%%%%%%%%%%%%%%%%%%%%%%%%%%%%
%%%%%%%%%%%%%%%%%%%%%%%%%%%%%%%%%%%%

Courant algebroids were introduced by Liu, Weinstein and Xu \cite{Liu_Weinstein_Xu} to interpret the bracket defined by Courant to study constraints on Dirac manifolds. In short, a Courant algebroid is a vector bundle $E \to M$ equipped with a symmetric nondegenerate bilinear form, together with a morphism of vector bundles $\rho:E \to TM$ and such that the space of sections $\Gamma(E)$ has the structure of a Leibniz algebra. All these data satisfy some compatibility conditions that we recall in Section \ref{section_1}. This is not the original definition introduced in \cite{Liu_Weinstein_Xu}, but an equivalent non-skew-symmetric version that uses what is known as Dorfman bracket instead of the Courant bracket.

 There is an alternative way to define Courant algebroids, introduced by Roytenberg \cite{royContemp}, which is the one that we consider in this paper. Courant algebroids can be described as degree $2$ symplectic graded manifolds together with a degree $3$ function $\Theta$ satisfying $\{\Theta, \Theta \}=0$, where $\{\cdot , \cdot \}$ is the graded Poisson bracket corresponding to the graded symplectic structure. This graded Poisson bracket is called \emph{big bracket} \cite{YKS92}. The morphism $\rho$ and the Dorfman bracket are recovered as derived brackets (see \cite{royContemp}).

When the Courant structure is defined on the Whitney sum $A\oplus A^*$  of a vector bundle $A$ and its dual, we have what we call a \emph{split} Courant algebroid. The Courant structure on  $A \oplus A^*$ can be the double of a Lie bialgebroid structure on $(A,A^*)$, the double of a quasi-Lie bialgebroid structure on $(A,A^*)$ or, more generally, the double of a proto-Lie bialgebroid structure on $(A,A^*)$\cite{roy}.

Besides Courant algebroids, the other relevant structures in this paper are $L_\infty$-algebras, also known as strongly homotopy Lie algebras. They were introduced by Lada and Stasheff \cite{lada_stasheff} and consist of
collections of $n$-ary brackets satisfying higher Jacobi identities. In the original definition of \cite{lada_stasheff}, the $n$-ary brackets are skew-symmetric, but in this paper we consider the equivalent definition where the brackets are graded symmetric.
Roytenberg and Weinstein \cite{Roytenberg_Weinstein} showed that to each Courant algebroid one can associate a Lie $2$-algebra and, recently, Lang, Sheng and Xu \cite{lang_sheng} proved a converse of this result.

In this paper we show that split Courant algebroids $A\oplus A^*$ are in a one-to-one correspondence with multiplicative curved $L_\infty$-algebra structures on $\Gamma(\wedge^\bullet A)[2]$. This extends other previous results. In 2002, Roytenberg \cite{roy} mentions that each split Courant algebroid which is the double of a quasi-Lie bialgebroid has an associated $L_\infty$-algebra defined on $\Gamma(\wedge^\bullet A)[2]$ and that the converse holds. No proof is given. In 2015, Fr\'{e}gier and Zambon \cite{fregier_zambon} proved that each split Courant algebroid which is the double of a proto-Lie bialgebroid determines a curved $L_\infty$-algebra structure on $\Gamma(\wedge^\bullet A^*)[2]$. The proof uses the higher derived brackets construction of Voronov \cite{voronov2}. We give an alternative and simpler proof that only uses the properties of the graded Poisson bracket, and we also prove the converse (Theorems  \ref{prop_M(theta)_is_pre_L_infty} and \ref{1-1curved_L_Courant}).

While we were finishing this paper,  Cueca and Mehta \cite{Cueca_M} established
an isomorphism between the sheaf of functions on a graded symplectic manifold of degree $2$ and the Keller-Waldmann algebra~\cite{KW15} of cochains on a vector bundle $E\to M$, equipped with a non-degenerated symmetric bilinear form $\langle\cdot, \cdot\rangle$. In this algebra, the $3-$cochains coincide with the pre-Courant structures on $(E,\langle\cdot, \cdot\rangle)$ and in this case the one-to-one correspondence was already known (see~\cite{royContemp}). In the case where $E=A\oplus A^*$ and $A\oplus A^*$ is equipped with the canonical pairing $\langle\cdot, \cdot\rangle$, the map $\mathcal{M}$ defined by equations~(\ref{definition_l0})-(\ref{definition_l3}) can be recovered from the isomorphism defined in~\cite{Cueca_M}.

Having established a one-to-one correspondence between split Courant algebroids and multiplicative curved $L_\infty$-algebras, it seemed interesting to discuss the behavior of Nijenhuis operators under this correspondence. Nijenhuis morphisms on Courant algebroids were initially considered in \cite{carinena_grab_marmo} and then revisited in \cite{grab}, under the graded manifold approach to Courant algebroids.
Regarding Nijenhuis forms on $L_\infty$-algebras, they were introduced in \cite{ALC15}. This notion also appears in \cite{sheng}, although with a simpler definition which turns out to be a particular case of the one in \cite{ALC15}. In this paper we consider the definition of \cite{ALC15}. Using the Lie $2$-algebra associated to each Courant algebroid according to \cite{Roytenberg_Weinstein}, some relations between Nijenhuis morphisms on Courant algebroids and Nijenhuis forms on Lie $2$-algebras were already established in \cite{ALC15}. In the current paper the approach is different since split Courant algebroids $A \oplus A^*$  are seen as graded manifolds, which is not the case in \cite{ALC15}, and the curved $L_\infty$-algebra structure  is defined on $\Gamma(\wedge^\bullet A)[2]$.

One of the advantages of viewing split Courant algebroids as graded manifolds, besides simpler and more efficient computations, is the relation with Lie algebroid structures on $A$. Indeed, we have that $(A\oplus A^*, \Theta=\mu)$ is a Courant algebroid if and only if $(A, \mu)$ is a Lie algebroid. Having this is mind, we characterize some know structures on Lie algebroids as Nijenhuis forms on $L_\infty$-algebra structures on $\Gamma(\wedge^\bullet A)[2]$.

Another type of operation that behave well under the one-to-one correspondence that we established, is the twisting on Courant algebroids and on $L_\infty$-algebras. The twisting of a split Courant algebroid by a bivector was defined in \cite{roy}, and the same operation can be done on $L_\infty$-algebras. In \cite{Getzler2009} it is shown that the twisting of a $L_\infty$-algebra by a degree zero element $\pi$ is an $L_\infty$-algebra provided that $\pi$ is a Maurer-Cartan element. In the case of a curved $L_\infty$-algebra, we show that $\pi$ no longer needs to be a Maurer-Cartan element.

 The paper is organized as follows. Section \ref{section_1} contains a brief review of the main notions concerning (pre-)Courant algebroids as well as  Nijenhuis morphisms on (pre-)Courant algebroids.
 In Section \ref{section_3} we recall the definition of curved $L_\infty$-algebras and of Nijenhuis forms on curved $L_\infty$-algebras.
 Section \ref{section_4} contains the main theorem, that establishes a one-to-one correspondence between split Courant algebroids and curved $L_\infty$-algebras.
In Section \ref{section_5} we show that the one-to-one correspondence preserves deformations by Nijenhuis operators. In particular, some Nijenhuis morphisms on Courant algebroids are characterized as Nijenhuis forms on curved $L_\infty$-algebras. Some well known structures on Lie algebroids are viewed as Nijenhuis form on $L_\infty$-algebras. In Section \ref{twisting by bivector} we discuss the twisting of a split Courant algebroid and of a curved $L_\infty$-algebra by $\pi \in \Gamma(\wedge^2 A)$ and we show that the one-to-one correspondence preserve these twisting operations. In Section \ref{Twisting and deformation} we combine the one-to-one correspondence with the operations of twisting by $\pi$ and deformation by a skew-symmetric vector-valued form on $\Gamma(\wedge^\bullet A)[2]$.
The commutative diagrams included along Sections \ref{section_4} to \ref{Twisting and deformation} can be combined to form a commutative cubic diagram, presented at the end of the paper.

%%%%%%%%%%%%%%%%%%%%%%%%%%%%%%%%%%%%%%%%%%%%%%%%%%%%%%%%%%%%%%%%%%%%%%%%
%%%%%%%%%%%%%%%%%%%%%%%%%%%%%%%%%%%%%%%%%%%%%%%%%%%%%%%%%%%%%%%%%%%%%%%%
%%%%%%%%%%%%%%%%%%%%%%%%%%%%%%%%%%%%%%%%%%%%%%%%%%%%%%%%%%%%%%%%%%%%%%%%
%
\section{Preliminaries on Courant algebroids and their Nijenhuis morphisms}         %%%%%%%%%%%%%%%%%%
\label{section_1}                                     %%%%%%%%%%%%%%%%%%
%%%%%%%%%%%%%%%%%%%%%%%%%%%%%%%%%%%%%%%%%%%%%%%%%%%%%%%%%%%%%%%%%%%%%%%%
%%%%%%%%%%%%%%%%%%%%%%%%%%%%%%%%%%%%%%%%%%%%%%%%%%%%%%%%%%%%%%%%%%%%%%%%
%%%%%%%%%%%%%%%%%%%%%%%%%%%%%%%%%%%%%%%%%%%%%%%%%%%%%%%%%%%%%%%%%%%%%%%%

%
%\subsection{Courant algebroids in supergeometric terms}  \label{subsection:1.1}
 In this section we recall the definition of Courant algebroid and how it can be seen as a $Q$-manifold,
following the approach of \cite{voronov,royContemp}. The notion of Nijenhuis morphism on a (pre)-Courant algebroid
is also recalled.

Let $E\to M$ be a vector bundle
equipped with a fibrewise non-degenerate symmetric bilinear form $\langle \cdot,\cdot \rangle$.
\begin{defn}\cite{ALC11} \label{def_courant}
 A \emph{pre-Courant} structure on $(E, \langle\cdot,\cdot\rangle)$ is a pair $(\rho, [\cdot,\cdot])$, where  $\rho:E\to TM$ is a morphism of vector bundles called the \emph{anchor},
 and $[\cdot,\cdot]:\Gamma(E)\times \Gamma(E)\to \Gamma(E)$ is a $\mathbb{R}$-bilinear bracket, called
the \emph{Dorfman bracket}, satisfying the relations
\begin{equation*} \label{pre_Courant1}
\rho(u)\cdot\langle v,w\rangle=\langle[u,v],w\rangle +  \langle v,[u,w]\rangle
\end{equation*}
and
\begin{equation*} \label{pre_Courant2}
\rho(u)\cdot \langle v,w\rangle=\langle u, [v,w]+ [w,v]\rangle,
\end{equation*}
for all $u,v,w \in \Gamma(E)$. The quadruple $(E, \langle\cdot,\cdot\rangle,\rho, [\cdot,\cdot])$ is a \emph{pre-Courant algebroid}.
\end{defn}
If a pre-Courant structure $(\rho, [\cdot,\cdot])$ satisfies the Jacobi identity,
$$[u,[v,w]] =[[u,v],w] + [v,[u,w]],$$
for all $u,v,w \in \Gamma(E)$, then the pair $(\rho, [\cdot,\cdot])$ is called a \emph{Courant} structure on \mbox{$(E,\langle\cdot,\cdot\rangle)$} and $(E, \langle\cdot,\cdot\rangle,\rho, [\cdot,\cdot])$ is a \emph{Courant algebroid}.

\

Next, we recall the notion of Nijenhuis morphism on a (pre-)Courant algebroid $(E,\langle \cdot,\cdot \rangle,\rho, [\cdot, \cdot])$.
 %with anchor and Dorfman bracket defined by (\ref{eq_derived_bracket_expressions}).
Given an endomorphism ${\mathcal{I}}:E \to E$, the transpose morphism ${\mathcal{I}}^*:E^*\simeq E \to E^*\simeq E$ is defined by $\langle {\mathcal{I}}^*u,v \rangle = \langle u,{\mathcal{I}}v \rangle$ for all $u,v \in E$. If ${\mathcal{I}}=-{\mathcal{I}}^*$, the morphism ${\mathcal{I}}$ is said to be \emph{skew-symmetric}.
For a skew-symmetric endomorphism $\mathcal{I}:E \to E$, we define a \emph{deformed} pre-Courant algebroid structure $(\rho_{\mathcal{I}}, [\cdot,\cdot]_{\mathcal{I}})$ on $(E, \langle\cdot,\cdot\rangle)$  by setting
\begin{equation}  \label{deformed}
\left\{
  \begin{array}{l}
    \rho_{\mathcal{I}}=\rho \circ {\mathcal{I}}\\
    {[u,v]_{\mathcal{I}}} =[{\mathcal{I}}u,v]+[u,{\mathcal{I}}v]-{\mathcal{I}}[u,v], \quad \forall u,v \in \Gamma(E).
  \end{array}
\right.\end{equation}

A skew-symmetric endomorphism ${\mathcal{I}}:E \to E$ on a pre-Courant algebroid $(E,\langle \cdot,\cdot \rangle,\rho,[\cdot,\cdot])$ is a {\em Nijenhuis morphism} if its Nijenhuis torsion ${\text{\Fontlukas T}}{\mathcal{I}}$ vanishes, where
\begin{equation*}\label{def_Nijenhuis_torsion}
    {\text{\Fontlukas T}}{\mathcal{I}}(u,v)=\frac{1}{2}\left([\mathcal{I}u, \mathcal{I}v]-\mathcal{I}\left([u,v]_{\mathcal{I}}\right)\right),
\end{equation*}
for all $u,v \in \Gamma(E)$. If $\mathcal{I}$ is a Nijenhuis morphism, then $(E,\langle \cdot,\cdot \rangle,\rho_{\mathcal{I}},[\cdot,\cdot]_{\mathcal{I}})$ is a Courant algebroid.

\

When the underlying vector bundle $E \to M$ of a (pre-)Courant algebroid is the Whitney sum $E=A \oplus A^*$ of a vector bundle $A \to M$ and its dual $A^* \to M$ we have a \emph{split} (pre-)Courant algebroid. The graded manifold approach of split \mbox{(pre-)}Courant algebroids will be extensively used in this paper, and so we briefly recall it.

Given a vector bundle $A \to M$, we denote by $A[m]$ the graded manifold obtained by shifting the fibre degree by $m$. The graded manifold $T^*[2]A[1]$ is equipped with a canonical symplectic structure which induces a graded Poisson bracket on its algebra of functions $\mathcal{F}:=C^\infty(T^*[2]A[1])$. This graded Poisson bracket is sometimes called the \emph{big bracket}.
 (see \cite{YKS92}).

Let us describe locally this Poisson algebra (see \cite{A10} for more details). Fix local coordinates $x_i, p^i,\xi_a, \theta^a$, $i \in \{1,\dots,n\}, a \in \{1,\dots,d\}$, in $T^*[2]A[1]$, where $x_i,\xi_a$ are local coordinates on $A[1]$ and $p^i, \theta^a$ are their associated moment coordinates. In these local coordinates, the Poisson bracket is given by
 $$ \{p^i,x_i\}=\{\theta^a,\xi_a\}=1,  \quad  i =1, \dots, n, \, \, a=1, \dots , d, $$
while all the remaining brackets vanish.

The Poisson algebra $(\mathcal{F}, \{\cdot, \cdot \})$ is endowed with an $(\mathbb{N}_0 \times \mathbb{N}_0)$-valued bidegree. We
define this bidegree (locally but it is well defined globally, see \cite{voronov, royContemp}) as follows: the coordinates on the base
manifold $M$, $x_i$, $i \in \{1,\dots,n\}$, have bidegree $(0,0)$, while the coordinates on the fibres, $\xi_a$, $a \in \{1,\dots,d\}$,
have bidegree $(0,1)$ and their associated moment coordinates, $p^i$ and $\theta^a$, have bidegree $(1,1)$ and $(1,0)$, respectively.
We denote by $\mathcal{F}^{k,l}$ the space of functions of bidegree $(k,l)$ and by $\mathcal{F}^{t}$ the space of functions of (total) degree $t$,
$$\mathcal{F}^{t}= \bigoplus _{k+l=t}\mathcal{F}^{k,l}.$$
Notice that $\mathcal{F}^0=C^\infty(M)$, $\mathcal{F}^{1,0}=\Gamma(A)$ and $\mathcal{F}^{0,1}=\Gamma(A^*)$.
The big bracket has bidegree $(-1,-1)$, i.e.,
$\{\mathcal{F}^{k_1,l_1},\mathcal{F}^{k_2,l_2}\}\subset \mathcal{F}^{k_1+k_2-1,l_1+l_2-1}$
and, for all $f,g \in \mathcal{F}^0=C^\infty(M)$ and $X+\alpha, Y+\beta \in \mathcal{F}^1=\Gamma(A\oplus A^*)$, we have
$$\{f,g\}=0,\; \; \{f, X+ \alpha \}=0 \; \; {\hbox{and}} \; \; \{X+\alpha,Y+\beta \}=\langle X+\alpha,Y+\beta \rangle,$$
where $\langle \cdot, \cdot \rangle$ stands for the usual pairing between $A$ and $A^*$,
$$\langle X+\alpha,Y+\beta \rangle:= \alpha(Y)+ \beta(X).$$

There is a one-to-one correspondence between pre-Courant structures $(\rho, [\cdot, \cdot])$ on $(A \oplus A^*, \langle\cdot,\cdot\rangle)$ and functions  $\Theta \in \mathcal{F}^3$. In other words, a pre-Courant structure on $(A \oplus A^*, \langle\cdot,\cdot\rangle)$ corresponds to a hamiltonian vector field $X_{\Theta}=\{ \Theta, \cdot \}$ on the graded manifold $T^*[2]A[1]$.
The anchor and Dorfman bracket associated to a given $\Theta\in \mathcal{F}^3$ are defined, for all $X+ \alpha,Y+ \beta \in \Gamma(A \oplus A^*)$ and $f \in C^\infty(M)$, by the derived bracket expressions
\begin{equation}\label{eq_derived_bracket_expressions}
  \rho(X+ \alpha)\cdot f=\{\{X+ \alpha,\Theta\},f\} \quad {\hbox{and}} \quad {[X+ \alpha,Y+ \beta]=\{\{X + \alpha,\Theta\},Y + \beta\}}.
\end{equation}

In \cite{royContemp, voronov} it is proved that there
is a one-to-one correspondence between Courant structures on $(A \oplus A^*,\langle \cdot,\cdot \rangle)$ and functions $\Theta \in \mathcal{F}^3$ such that the hamiltonian vector field $X_{\Theta}$ on $T^*[2]A[1]$ is a homological vector field, i.e., $\{\Theta,\Theta\}=0$. Thus, a Courant algebroid $(A \oplus A^*, \langle \cdot,\cdot\rangle, \Theta)$ corresponds to a $Q$-manifold $(T^*[2]A[1], X_{\Theta})$.

In what follows, a split (pre-)Courant algebroid will be denoted simply by $(A \oplus A^*,\Theta)$.

\

A (pre-)Courant structure $\Theta \in \mathcal{F}^3$ can be decomposed using the bidegrees, as follows:
\begin{equation} \label{Theta}
\Theta=\psi + \gamma + \mu + \phi,
\end{equation}
with $\psi \in \mathcal{F}^{3,0}=\Gamma(\wedge^3 A), \gamma \in \mathcal{F}^{2,1}, \mu \in \mathcal{F}^{1,2}$ and  $\phi \in \mathcal{F}^{0,3}=\Gamma(\wedge^3 A^*)$.
We recall from~\cite{roy} that, when $\psi=\gamma = \phi =0$, $\Theta$ is a Courant structure on $A \oplus A^*$ \emph{if and only if} $(A,\mu)$ is a Lie algebroid.
When $\psi = \phi =0$, $\Theta$ is a Courant structure on $A \oplus A^*$ \emph{if and only if} $((A,A^*), \mu, \gamma)$ is a Lie bialgebroid and when $\phi=0$ (resp. $\psi=0$), $\Theta$ is a Courant structure on $A \oplus A^*$ \emph{if and only if} $((A,A^*), \mu, \gamma, \psi)$ (resp. $((A^*,A), \gamma, \mu, \phi)$) is a quasi-Lie bialgebroid. In the more general case, $\Theta=\psi + \gamma + \mu + \phi$ is a Courant structure \emph{if and only if} $((A,A^*), \mu, \gamma, \psi, \phi)$  is a proto-Lie bialgebroid. In this general case,
\begin{equation}\label{eq_(theta,theta)=0}
\{\Theta,\Theta\}=0 \,\, \Leftrightarrow \,\, \begin{cases}
 \{\gamma, \psi \}=0\\
        \{\gamma, \gamma \}+ 2\{\mu, \psi \}=0\\
        \{\mu, \gamma \}+ \{\psi,\phi \}=0\\
        \{\mu, \mu \}+2 \{\gamma, \phi\} =0\\
        \{ \mu, \phi \}=0.
    \end{cases}
\end{equation}

\medskip

%%%%%%%%%%%%%%%%%%%%%%%%%%%%%%%%%%%%%%%%%%%%%%%%%%%%%%%%%%%%%%%%%%%%%%%%
%%%%%%%%%%%%%%%%%%%%%%%%%%%%%%%%%%%%%%%%%%%%%%%%%%%%%%%%%%%%%%%%%%%%%%%%
Now, we shall see what is the function on $\mathcal{F}^3$ corresponding to the deformed \mbox{(pre-)}Courant structure (\ref{deformed}) on $A \oplus A^*$. A skew-symmetric endomorphism on $A \oplus A^*$, $J:A \oplus A^* \to A \oplus A^*$, is of the type
\begin{equation}  \label{skewJ}
J= \left(
\begin{array}{cc}
 N & \pi^{\sharp} \\
\omega^{\flat} & -N^*
\end{array}
\right),
\end{equation}
with $N: A \to A, \pi \in \Gamma(\wedge^2 A)$, $\omega \in \Gamma(\wedge^2 A^*)$ and where $N^*:A^*\to A^*$, $\pi^{\sharp}:A^*\to A$, $\omega^{\flat}:A\to A^*$ are defined by
$$\left\{
  \begin{array}{l}
     \langle N^* \alpha, X\rangle=\langle \alpha, N X\rangle\\
     \langle \pi^{\sharp}(\alpha), \beta\rangle=\pi(\alpha, \beta)\\
     \langle \omega^{\flat}(X), Y\rangle=\omega(X,Y),
  \end{array}
\right.$$
for all $X,Y \in \Gamma(A)$ and $\alpha, \beta \in \Gamma(A^*)$.\footnote{We use the same notation for the maps induced on the space of sections $\Gamma(A)$ and $\Gamma(A^*)$.}

We have
$$J(X+\alpha)=\bb{X+\alpha}{\pi+N+\omega},$$
so the morphism $J$ corresponds to the function $\pi+N+\omega\in \Gamma(\wedge^2 (A \oplus A^*))\subset C^\infty(T^*[2]A[1])$, that we also denote by $J$.

The deformation of the (pre-)Courant structure $\Theta$ by $J$ is the function $\Theta_{J}= \Theta_{\pi+N+\omega}:=\{ \pi+ N+ \omega, \Theta \}\in\mathcal{F}^{3}$, that corresponds to
$(\rho_{J}, [\cdot,\cdot]_{J})$  (via (\ref{eq_derived_bracket_expressions})).

Given a Courant structure $\Theta$ on $E$ and a skew-symmetric endomorphism $J$, $\Theta_J$ is not necessarily a Courant structure. However, f the Nijenhuis torsion ${\text{\Fontlukas T}}_{\Theta} J$ of $J$ vanishes, then $(E, \Theta_J )$ is a Courant algebroid. Moreover, since
\begin{equation*}
\{ \Theta_J, \Theta_J \}=\{ (\Theta_{J})_{J}, \Theta\},
\end{equation*}
where $(\Theta_{J})_{J}$ denotes the deformation of $\Theta_{J}$ by $J$,
a necessary and sufficient condition for $(E, \Theta_J )$ to be a Courant algebroid is that $\{ (\Theta_{J})_{J}, \Theta\}=0$ \cite{YKS11}.

When $J$  satisfies ${J}^2= \lambda\, {\rm id}_{A \oplus A^*}$, for some $\lambda \in \Rr$,
the Nijenhuis torsion of $J$ is given by \cite{grab,A10}
\begin{equation} \label{supergeometric_torsion}
{\text{\Fontlukas T}}_{\Theta} J= \frac{1}{2}((\Theta_{J})_{J}-\lambda \Theta).
\end{equation}

In this case, and if $\Theta$ is a Courant structure, \begin{equation*}
\{ \Theta_J, \Theta_J \}= 2 \{ {\text{\Fontlukas T}}_{\Theta} J, \Theta\}.
\end{equation*}

%%%%%%%%%%%%%%%%%%%%%%%%%%%%%%%%%%%%%%%%%%%%%%%%%%%%%%%%%%%%%%%%%%%%%%%%%%%%%%%%%%%%%%%%%%%
%%%%%%%%%%%%%%%%%%%%%%%%%%%%%%%%%%%%%%%%%%%%%%%%%%%%%%%%%%%%%%%%%%%%%%%%%%%%%%%%%%%%%%%%%%%
%%%%%%%%%%%%%%%%%%%%%%%%%%%%%%%%%%%%%%%%%%%%%%%%%%%%%%%%%%%%%%%%%%%%%%%%%%%%%%%%%%%%%%%%%%%

\section{Review on $L_\infty$-algebras and Nijenhuis forms}
\label{section_3}
%
%%%%%%%%%%%%%%%%%%%%%%%%%%%%%%%%%%%%%%%%%%%%%%%%%%%%%%%%%%%%%%%%%%%%%%%%
%%%%%%%%%%%%%%%%%%%%%%%%%%%%%%%%%%%%%%%%%%%%%%%%%%%%%%%%%%%%%%%%%%%%%%%%
In this section we recall the definitions of curved (pre-)$L_\infty$-algebra and Nijenhuis form on an $L_\infty$-algebra, following \cite{ALC15}. For the definition of an $L_\infty$-algebra we consider graded symmetric brackets, which is not the case in the original definition introduced in \cite{lada_stasheff}. Both definitions are equivalent, and the equivalence is given by the so-called \emph{d\'{e}calage isomorphism} (see \cite{voronov2,ALC15} for more details).

In what follows, we consider graded vector spaces with all components of finite dimension.
\begin{defn}
A \emph{curved pre-$L_\infty$-algebra $(\mathcal{L},\mathcal{l})$} is a graded vector space $\mathcal{L}=\bigoplus_{i \in \mathbb{Z}}\mathcal{L}_i$ together with a family of symmetric vector-valued forms (brackets) $\mathcal{l}_i:\otimes^i\mathcal{L}\to \mathcal{L}, i\geq 0$, of degree $1$. For $i=0, \, \mathcal{l}_0 \in \mathcal{L}_1$. The term $\mathcal{l}_0$ is called the \emph{curvature}. We write $\mathcal{l}=\sum_{i\geq 0}\mathcal{l}_i$.

The pair $(\mathcal{L},\mathcal{l})$ is called a \emph{curved $L_\infty$-algebra} if the generalized Jacobi identity is satisfied:
\begin{equation}  \label{gen_Jacobi_id}
\displaystyle{\sum_{i+j=n+1} \, \sum_{\sigma \in Sh(i,j-i)}\epsilon(\sigma)\mathcal{l}_j(\mathcal{l}_i(X_{\sigma(1)},\cdots,X_{\sigma(i)}),\cdots,X_{\sigma(n)})=0}
\end{equation}
 for all $n \in \mathbb{N}_0$, where $Sh(i,j-i)$ stands for the set of $(i, j-i)$-unshuffles and $\epsilon(\sigma)$ is the (graded commutative) Koszul sign defined by
$$X_{\sigma(1)}\otimes\ldots\otimes X_{\sigma(n)}=\epsilon(\sigma)\,X_1\otimes\ldots\otimes X_{n},$$
for all $X_1,\ldots, X_{n} \in \mathcal{L}$.
When the curvature vanishes, i.e. $\mathcal{l}_0=0$, $(\mathcal{L},\mathcal{l})$ is simply called an \emph{ $L_\infty$-algebra}.

\end{defn}
%\vspace{-3em}

For $k \geq 0$, we denote by $S^k(\mathcal{L}^*)\otimes \mathcal{L}$ the space of symmetric vector-valued $k$-forms on the graded vector space $\mathcal{L}$, i.e., graded symmetric $k$-linear maps on $\mathcal{L}$, and we set
 $$S^{\bullet}(\mathcal{L}^*)\otimes \mathcal{L}= \bigoplus_{k \geq 0}S^k(\mathcal{L}^*)\otimes \mathcal{L}.$$
 For $k=0$, $S^0(\mathcal{L}^*)\otimes \mathcal{L}$ is isomorphic to $\mathcal{L}$.
%\medskip

The insertion operator of a symmetric vector-valued $k$-form $K$ is an operator
$$\imath_{K}: S^{\bullet}(\mathcal{L}^*)\otimes \mathcal{L} \rightarrow S^{\bullet}(\mathcal{L}^*)\otimes \mathcal{L} $$
defined by
$$\imath_K H(X_1,\ldots, X_{k+h-1})=\sum_{\sigma \in Sh(k,h-1)}\epsilon(\sigma) H\left(K(X_{\sigma(1)},\ldots, X_{\sigma(k)}),\ldots, X_{\sigma(k+h-1)}\right),$$
for all $H \in S^{h}(\mathcal{L}^*)\otimes \mathcal{L}$ and $X_1,\ldots, X_{k+h-1} \in \mathcal{L}$. If $H \in S^{0}(\mathcal{L}^*)\otimes \mathcal{L} \simeq \mathcal{L}, \imath_K H=0$.

Given a symmetric vector-valued $k$-form $K\in S^k(\mathcal{L}^*)\otimes \mathcal{L}$ and a symmetric vector-valued $h$-form $H\in S^h(\mathcal{L}^*)\otimes \mathcal{L}$, the \emph{Richardson-Nijenhuis bracket} of $K$ and $H$ is the symmetric vector-valued $(k+h-1)$-form $[K,H]$ on $\mathcal{L}$, given by
\begin{equation}\label{definition_RNbracket}
  [K,H]=\imath_K H -(-1)^{\overline{K}\,\overline{H}}\imath_H K,
\end{equation}
where $\overline{K}$ is the degree of $K$ as a graded map, that is $K(X_1,\ldots, X_k)\in \mathcal{L}_{x_1+\ldots+x_k+\bar{K}}$, for all $X_i\in \mathcal{L}_{x_i}, i=1,\ldots,k$.
The pair $(S^{\bullet}(\mathcal{L}^*)\otimes \mathcal{L}, [\cdot, \cdot])$ is a graded skew-symmetric Lie algebra.

Curved $L_\infty$-algebras can be characterized using the Richardson-Nijenhuis bracket.
\begin{prop}\cite{ALC15}
A curved pre-$L_\infty$-algebra $(\mathcal{L}, \mathcal{l})$ is a  curved $L_\infty$-algebra if and only if $[\mathcal{l}, \mathcal{l}]=0$.
\end{prop}

\

Assume that there exists an associative graded commutative algebra structure of degree zero on $\mathcal{L}$, denoted by $\wedge$. A vector-valued $k$-form $K \in S^{k}(\mathcal{L}^*)\otimes \mathcal{L}$ is said to be a \emph{multiderivation symmetric vector-valued $k$-form} if
$$K(X_1, \cdots , X_{k-1}, Y \wedge Z)=K(X_1, \cdots , X_{k-1}, Y) \wedge Z + (-1)^{yz}K (X_1, \cdots , X_{k-1}, Z) \wedge Y,$$
for all $X_1, \cdots X_{k-1}\in \mathcal{L}$, $Y \in \mathcal{L}_y$ and $Z  \in \mathcal{L}_z$.

The space of all multiderivation symmetric vector-valued forms on $\mathcal{L}$ is a graded Lie subalgebra of $(S^{\bullet}(\mathcal{L}^*)\otimes \mathcal{L}, [\cdot, \cdot])$.

 A curved $L_\infty$- algebra $(\mathcal{L}, \mathcal{l})$ is called \emph{multiplicative} if all the brackets are multiderivations. Multiplicative (curved) $L_\infty$- algebras are also called (curved) \emph{$P_\infty$- algebras} \cite{CF07}.  They can be viewed as a symmetric version of $G_\infty$-algebras.

\medskip

Given a curved $L_\infty$-structure $\mathcal{l}$ and a symmetric vector-valued form of degree zero, $\mathcal {n}$, on a graded vector space, we call $[\mathcal{n}, \mathcal{l}]$ the \emph{deformation of $\mathcal{l}$ by $\mathcal{n}$} and denote the deformed structure by $\mathcal{l}_{\mathcal{n}}:=[\mathcal{n}, \mathcal{l}]$.

In general, $\mathcal{l}_{\mathcal{n}}$ is not a curved $L_\infty$-algebra. However, under some conditions on $\mathcal{n}$ we can guarantee it, namely when $\mathcal{n}$ is a Nijenhuis vector-valued form. Let us now
recall the definition of Nijenhuis vector-valued form on an $L_\infty$-algebra, introduced in \cite{ALC15}.

\begin{defn} \label{def_Nijenhuis}
  Let $(\mathcal{L},\mathcal{l})$ be a curved pre-$L_\infty$-algebra. A symmetric vector-valued form $\mathcal{n}$ on $\mathcal{L}$, of degree zero, is called a \emph{Nijenhuis form} on $(\mathcal{L},\mathcal{l})$ if there exists a vector-valued form $\mathcal{k}$ of degree zero, such that
  $$[\mathcal{n},[\mathcal{n},\mathcal{l}]]=[\mathcal{k}, \mathcal{l}]\quad \text{and} \quad [\mathcal{n},\mathcal{k}]=0.$$
  Such a vector-valued form $\mathcal{k}$ is called a \emph{square} of $\mathcal{n}$.
\end{defn}

In the forthcoming sections we will often use the so-called \emph{Euler map}.
  Given a graded vector space $\mathscr{L}=\bigoplus_{i \in \mathbb{Z}}\mathcal{L}_i$, the Euler map $\mathcal{E}: \mathcal{L} \to \mathcal{L}$ is a linear map of degree zero defined by
  \begin{equation}\label{eqn_EulerMap}
    \mathcal{E}(P)=pP,
  \end{equation}
  for all homogeneous elements $P \in \mathcal{L}_p$.

%%%%%%%%%%%%%%%%%%%%%%%%%%%%%%%%%%%%%%%%%%%%%%%%%%%%%%%%%%%%%%%%%%%%%%%%%%%%%%%%%%%%%%%%%%%
%%%%%%%%%%%%%%%%%%%%%%%%%%%%%%%%%%%%%%%%%%%%%%%%%%%%%%%%%%%%%%%%%%%%%%%%%%%%%%%%%%%%%%%%%%%
%%%%%%%%%%%%%%%%%%%%%%%%%%%%%%%%%%%%%%%%%%%%%%%%%%%%%%%%%%%%%%%%%%%%%%%%%%%%%%%%%%%%%%%%%%%

\section{From Courant algebroids to $L_\infty$-algebras and back}
\label{section_4}

In this section we prove a theorem that generalizes a result initially established by Roytenberg \cite{roy} in the case of a split Courant algebroid which is the double of a quasi-Lie bialgebroid, and then extended by Fr\'{e}gier and Zambon \cite{fregier_zambon} to the case where the Courant structure is the double of a proto-bialgebroid. A result similar to the one in \cite{fregier_zambon} was obtained by Gualtieri, Matviichuk and Scott \cite{gualtieri_al}. In all cases, given a split Courant algebroid structure, a (curved) $L_\infty$-algebra is constructed. Our theorem includes the converse and the proof uses a technique different from the one in \cite{fregier_zambon}.

%%%%%%%%%%%%%%%%%%%%%%%%%%%%%%%%%%%%%%%%%%%%%%%%%%%%%%%%%%%%%%%%%%%%%%%%
%%%%%%%%%%%%%%%%%%%%%%%%%%%%%%%%%%%%%%%%%%%%%%%%%%%%%%%%%%%%%%%%%%%%%%%%

Let $(A \oplus A^*, \Theta)$ be a pre-Courant algebroid where $\Theta \in \mathcal{F}^3$ can be decomposed using the bidegrees as in (\ref{Theta}):
\begin{equation*}\label{Decomposition_Theta}
\Theta=\psi + \gamma + \mu + \phi.
\end{equation*}

Set $L=\Gamma(\wedge^{\bullet} A)[2]$. Thus $L=\sum_{i=-2}^{\infty} L_{i}$ is a graded vector space where  $L_{-2}=C^\infty(M), L_{-1}=\Gamma(A)$ and $L_{i}=\Gamma(\wedge^{i+2} A)$, $i\geq 0$.

Let us consider the map

\begin{equation*}\label{functor_M}
\mathcal{M}: \xymatrixcolsep{3pc}\xymatrix@R=1pc{
    \mathcal{F}^{3} \ar@{->}[r] &  S^{\bullet}(L^*)\otimes L\\
    \Theta=\psi + \gamma + \mu + \phi \ar@{|->}[r]&  l=l_0+l_1+l_2+l_3
    }
\end{equation*}
where
\begin{itemize}
\item  $\mathcal{M}(\psi)=l_0 \in L_1=\Gamma(\wedge^3 A)$ is defined by
\begin{equation}\label{definition_l0}
l_0=\psi;
\end{equation}

\item  $\mathcal{M}(\gamma)=l_1 \in S^1(L^*)\otimes L$  is defined by
\begin{equation}\label{definition_l1}
l_1(P)=\bb{\gamma}{P};
\end{equation}
 %for all homogeneous element $P= P_1\wedge\ldots\wedge P_{p}$ in $L_{p-2}=\Gamma(\wedge^p A)$;\\
\item  $\mathcal{M}(\mu)=l_2 \in S^2(L^*)\otimes L$ is defined by
    \begin{equation}\label{definition_l2}
        l_2(P,Q)=\bb{\bb{\mu}{P}}{Q};
    \end{equation}
\item  $\mathcal{M}(\phi)=l_3 \in S^3(L^*)\otimes L$ is defined by
    \begin{equation}\label{definition_l3}
        l_3(P,Q,R)=\bb{\bb{\bb{\phi}P}Q}R,
    \end{equation}
\end{itemize}
for all $P,Q,R \in \Gamma(\wedge^{\bullet} A)$.

\

\begin{thm}\label{prop_M(theta)_is_pre_L_infty}
  The map $\mathcal{M}$, defined by Equations (\ref{definition_l0})-(\ref{definition_l3}), establishes a one-to-one correspondence between pre-Courant structures on $A\oplus A^*$  and multiplicative curved pre-$L_\infty$-algebra structures $l=l_0+l_1+l_2+l_3$ on $L=\Gamma(\wedge^{\bullet} A)[2]$.
  \end{thm}

Before proving Theorem~\ref{prop_M(theta)_is_pre_L_infty}, let us recall Lemma 3.1.6 from \cite{A10}.

\begin{lem}\label{lema_da tese}
  Consider $F \in \mathcal{F}^{r,s}$, with $s>0$. If $F$ satisfies $\bb{F}{X}=0$, for all $X \in \Gamma(A)$, then $F=0$.
\end{lem}

Now let us prove Theorem \ref{prop_M(theta)_is_pre_L_infty}

\begin{proof}[Proof of Theorem \ref{prop_M(theta)_is_pre_L_infty}]
   Let $\Theta=\psi + \gamma + \mu + \phi$ be a pre-Courant structure on $A\oplus A^*$. First, let us prove that $l=\mathcal{M}(\Theta)$, defined by Equations (\ref{definition_l0})-(\ref{definition_l3}), is a multiplicative curved pre-$L_\infty$-algebra structure. The fact that $l$ is multiplicative is a direct consequence of the definition of $l$ and the Leibniz rule for the big bracket $\{\cdot, \cdot \}$. The remaining part of the statement claims that $l=\sum_{i=0}^3 l_i$ is a graded symmetric linear map of degree $1$. This is immediate due to the definition of $l=\mathcal{M}(\Theta)$ and to the properties of the big bracket in $C^\infty(T^*[2]A[1])$. For example, let us check explicitly that $l_2$ is a graded symmetric map $S^2(L)\to L$ of degree $1$.

  For all $P \in L_p=\Gamma(\wedge^{p+2} A)$ and $Q\in L_q=\Gamma(\wedge^{q+2} A)$, using the Jacobi identity of the big bracket, we have
  \begin{align*}
     l_2(Q,P)&=\bb{\bb{\mu}{Q}}{P}=\bb{\mu}{\bb{Q}{P}}+(-1)^{(p+2)(q+2)}\bb{\bb{\mu}{P}}{Q}=\\
    &=(-1)^{pq}\bb{\bb{\mu}{P}}{Q}=(-1)^{pq}\,l_2(P,Q),
  \end{align*}
  which proves that $l_2$ is a graded symmetric map.
Furthermore, in $C^\infty(T^*[2]A[1])$, the big bracket is a map of bidegree $(-1,-1)$ and the elements  $\mu, P$ and $Q$ have bidegrees $(1,2),(p+2,0)$ and $(q+2,0)$, respectively. Thus, $\bb{\bb{\mu}{P}}{Q}$ has bidegree $$((1,2)+(p+2,0)+(-1,-1))+(q+2,0)+(-1,-1)=(p+q+3,0),$$ which means that $$l_2(P,Q)=\bb{\bb{\mu}{P}}{Q}\in \Gamma(\wedge^{p+q+3} A)=L_{p+q+1}.$$ Then $l_2$ is a map of degree $1$.

  Conversely, given a multiplicative curved pre-$L_\infty$-algebra structure $l=l_0+l_1+l_2+l_3$ on $L=\Gamma(\wedge^{\bullet} A)[2]$, let us prove that there is an unique $\Theta_n \in \mathcal{F}^{3-n,n}$ such that $\mathcal{M}(\Theta_n)=l_n$, for each $n=0,1,2,3$.
\begin{itemize}
\item For $n=0$, we have $\Theta_0=l_0\in \Gamma(\wedge^3 A)=\mathcal{F}^{3,0}$.
\item For $n=1$, we need to define $\Theta_1 \in \mathcal{F}^{2,1}$, such that $\mathcal{M}(\Theta_1)=l_1$, i.e, such that,
\begin{equation}\label{eq_aux_2}
\bb{\Theta_1}{P}=l_1(P),\quad  \forall P\in \Gamma(\wedge^\bullet A).
\end{equation}
 We claim\footnote{Recall that $\Gamma(\wedge^\bullet A)$ is generated by $C^\infty(M)$ and $\Gamma (A)$.} that Equation (\ref{eq_aux_2}) defines explicitly an unique $\Theta_1\in \mathcal{F}^{2,1}$. Indeed, locally, on coordinates $(x_i, p^i,\xi_a, \theta^a)$, such an element is written  as $\Theta_1=A_{ia}(x)p^i\theta^a + B_{ab}^c(x) \theta^a\theta^b\xi_c$ and, using Equation (\ref{eq_aux_2}), its coefficients are determined (apart from signs that depend on conventions) by $l_1$, as follows:
  $$\left\{
    \begin{array}{l}
      A_{ia}=\pm\bb{\bb{\Theta_1}{x_i}}{\xi_a}=\pm<l_1(x_i), \xi_a>\\
      B_{ab}^c=\pm\frac{1}{2}\bb{\bb{\bb{\Theta_1}{\theta^c}}{\xi_a}}{\xi_b}=\pm\frac{1}{2}<l_1(\theta^c), \xi_a\wedge\xi_b>.
    \end{array}
  \right.$$
Thus, the existence of $\Theta_1$ satisfying Equation (\ref{eq_aux_2}) is guaranteed. Furthermore, we can not have two elements $\Theta_1$ and $\Theta_1'$ satisfying Equation (\ref{eq_aux_2}) because Lemma \ref{lema_da tese} would imply that $\Theta_1-\Theta_1'=0$.
\item Analogously, for $n=2$, we need to define $\Theta_2 \in \mathcal{F}^{1,2}$, such that $\mathcal{M}(\Theta_2)=l_2$, i.e., such that
\begin{equation}\label{eq_aux_3}
\bb{\bb{\Theta_2}{P}}{Q}=l_2(P, Q),\quad  \forall P, Q\in \Gamma(\wedge^\bullet A).
\end{equation}
Locally, $\Theta_2=C_{i}^a(x)p^i\xi_a + D_{c}^{ab}(x) \xi_a\xi_b\theta^c$ and, using Equation (\ref{eq_aux_3}), the coefficients are determined by $l_2$ as follows:
  $$\left\{
    \begin{array}{l}
      C_{i}^a=\pm\bb{\bb{\Theta_2}{x_i}}{\theta^a}=\pm\,l_2(x_i, \theta^a)\\
      D_{c}^{ab}=\pm\frac{1}{2}\bb{\bb{\bb{\Theta_2}{\theta^a}}{\theta^b}}{\xi_c}=\pm\frac{1}{2}<l_2(\theta^a, \theta^b), \xi_c>.
    \end{array}
  \right.$$
\item Finally, for $n=3$, $\Theta_3 \in \mathcal{F}^{0,3}=\Gamma(\wedge^3 A^*)$ is a $3$-form and condition $\mathcal{M}(\Theta_3)=l_3$ implies that
$$\bb{\bb{\bb{\Theta_3}{X}}{Y}}{Z}=l_3(X, Y, Z),$$
for all $X,Y, Z \in \Gamma(A)$, and this defines uniquely $\Theta_3$.
\end{itemize}
\end{proof}

\begin{thm} \label{1-1curved_L_Courant}
  Let $\Theta \in \mathcal{F}^3$ be a pre-Courant structure on $A\oplus A^*$ and $l=\mathcal{M}(\Theta)$ its corresponding multiplicative curved pre-$L_\infty$-algebra structure on $L=\Gamma(\wedge^{\bullet} A)[2]$. Then, the following assertions are equivalent:
  \begin{enumerate}
    \item $(A\oplus A^*, \Theta)$ is a Courant algebroid;
    \item $(L,l)$ is a multiplicative curved $L_\infty$-algebra.
  \end{enumerate}
\end{thm}
\begin{proof}
  The generalized Jacobi identity (\ref{gen_Jacobi_id}) satisfied by $l$ corresponds exactly to the different conditions we obtained in (\ref{eq_(theta,theta)=0}), after splitting the condition $\bb{\Theta}{\Theta}=0$ using bidegree.
  Indeed, for $n=0$, we have
  \begin{equation}\label{l_1l_0}
  l_1(l_0)=0 \Leftrightarrow \bb{\gamma}{\psi}=0
  \end{equation}
   while for $n=1$ and for all $P \in L$,
  \begin{align} \label{l_2l_0}
     l_2(l_0, P)+l_1(l_1(P))=0& \Leftrightarrow \bb{\bb{\mu}{\psi}}{P}+\bb{\gamma}{\bb{\gamma}{P}}=0 \nonumber \\
    &\Leftrightarrow \bb{\bb{\mu}{\psi}+\frac{1}{2}\bb{\gamma}{\gamma}}{P}=0 \nonumber \\
    &\Leftrightarrow \bb{\mu}{\psi}+\frac{1}{2}\bb{\gamma}{\gamma}=0,
  \end{align}
where the last equivalence follows from Lemma \ref{lema_da tese}.
  For $n=2$ we have, for all $P \in L_p$ and $Q \in L_q$,
  \begin{align} \label{l_3l_0}
    &l_3(l_0, P, Q)+l_2(l_1(P),Q) + (-1)^{pq}l_2(l_1(Q),P) + l_1(l_2(P,Q))=0 \nonumber \\
    \Leftrightarrow\quad & \bb{\bb{\bb{\phi}{\psi}}{P}}{Q}+ \bb{\bb{\mu}{\bb{\gamma}{P}}}{Q} + (-1)^{pq} \bb{\bb{\mu}{\bb{\gamma}{Q}}}{P} \nonumber \\ & + \bb{\gamma}{\bb{\bb{\mu}{P}}{Q}}=0 \nonumber \\
    \Leftrightarrow\quad & \bb{\bb{\bb{\phi}{\psi}+\bb{\mu}{\gamma}}{P}}{Q}=0 \nonumber \\
    \Leftrightarrow\quad & \bb{\phi}{\psi}+\bb{\mu}{\gamma}=0.
  \end{align}
Equations (\ref{l_1l_0}), (\ref{l_2l_0}) and (\ref{l_3l_0}) are precisely the first, second and third equations on the right side of (\ref{eq_(theta,theta)=0}).

  For $n=3, 4$ and $5$, since more terms are involved, computations are rather cumbersome but straightforward and only use the properties of the big bracket (essentially Jacobi identity). Computations for $n=3$ and $4$ lead to the last two equations on the right side of (\ref{eq_(theta,theta)=0}). For $n=5$, we prove that for all $P,Q,R,S,T \in L$,
  $$\circlearrowleft l_3(l_3(P,Q,R),S,T)=\frac{1}{2}\bb{\bb{\bb{\bb{\bb{\bb{\phi}{\phi}}{P}}{Q}}{R}}{S}}{T},$$
  where $\circlearrowleft$ stands for the sum of the ten terms corresponding to the graded $(3,2)$-unshuffled permutations of the set $\{P,Q,R,S,T\}$. This condition is trivially satisfied because $\bb{\phi}{\phi}=0$, for bidegree reasons, for any $\phi \in \Gamma(\wedge^3 A^*)$.
\end{proof}

Notice that the roles of the vector bundle $A$ and its dual $A^*$ can be reversed everywhere in this section, since $(A \oplus A^*, \Theta)$ is a Courant algebroid if and only if
$(A^* \oplus A, \Theta)$ is a Courant algebroid \cite{roy}. As a consequence, in Theorem \ref{1-1curved_L_Courant}, instead of considering the graded vector space $L=\Gamma(\wedge^{\bullet} A)[2]$, we can take  $\mathfrak{L}:=\Gamma(\wedge^{\bullet} A^*)[2]$ and define the following graded symmetric brackets of degree $1$:
\begin{equation}\label{lambda_i}
\left\{
  \begin{array}{l}
  %  \begin{align*}
         \lambda_0=\phi \\
         \lambda_1(\alpha)=\bb{\mu}{\alpha}   \\
         \lambda_2(\alpha,\beta)=\bb{\bb{\gamma}{\alpha}}{\beta}  \\
         \lambda_3(\alpha,\beta,\eta)=\bb{\bb{\bb{\psi}\alpha}\beta}\eta
  % \end{align*}
  \end{array}
\right.
\end{equation}

\noindent for all $\alpha, \beta, \eta \in \Gamma(\wedge^{\bullet} A^*)$. Set $\lambda= \sum_{i=0}^{3} \lambda_i$.

 Next corollary summarizes what we have proved so far.

\begin{cor} \label{curved_L*_Courant}
The following assertions are equivalent:
\begin{enumerate}
\item
$(A \oplus A^*, \Theta)$ is a Courant algebroid;
\item
$(A^* \oplus A, \Theta)$ is a Courant algebroid;
\item
$(L,l)$ is a curved $L_\infty$-algebra;
\item
$(\mathfrak{L},\lambda)$ is a curved $L_\infty$-algebra.
\end{enumerate}
\end{cor}

\begin{rem}
In \cite{fregier_zambon} it is proved that ($i$) (or ($ii$)) implies ($iv$). The technique used in \cite{fregier_zambon} to obtain the curved $L_\infty$-algebra structure is the Voronov's higher derived brackets \cite{voronov2}.  Although we also use the derived brackets construction, our proof is simpler and only uses the properties of the graded Poisson bracket. See also \cite{gualtieri_al} for the case of exact Courant algebroids.
\end{rem}

Having established Corollary \ref{curved_L*_Courant}, we can proceed over the next sections either with the curved $L_\infty$-algebra $(L=\Gamma(\wedge^\bullet A)[2], l)$ or with the curved $L_\infty$-algebra $(\mathfrak{L}=\Gamma(\wedge^\bullet A^*)[2],\lambda)$. We will continue with $(L,l)$, but one should have in mind that all the forthcoming results have their \emph{dual version} if we would consider $(\mathfrak{L},\lambda)$ instead of $(L,l)$.

%%%%%%%%%%%%%%%%%%%%%%%%%%%%%%%%%%%%%%%%%%%%%%%%%%%%%%%%%%%%%%%%%%%%%%%%%%%%%%%%%%%%%%%%%%%
%%%%%%%%%%%%%%%%%%%%%%%%%%%%%%%%%%%%%%%%%%%%%%%%%%%%%%%%%%%%%%%%%%%%%%%%%%%%%%%%%%%%%%%%%%%
%%%%%%%%%%%%%%%%%%%%%%%%%%%%%%%%%%%%%%%%%%%%%%%%%%%%%%%%%%%%%%%%%%%%%%%%%%%%%%%%%%%%%%%%%%%

\section{Nijenhuis on Courant algebroids and on $L_\infty$-algebras}
\label{section_5}
%
%%%%%%%%%%%%%%%%%%%%%%%%%%%%%%%%%%%%%%%%%%%%%%%%%%%%%%%%%%%%%%%%%%%%%%%%
%%%%%%%%%%%%%%%%%%%%%%%%%%%%%%%%%%%%%%%%%%%%%%%%%%%%%%%%%%%%%%%%%%%%%%%%
In this section, to each skew-symmetric endomorphism on $A \oplus A^*$ we associate a vector-valued form of degree zero on $\Gamma(\wedge^\bullet A)[2]$ and we analyse how the induced deformations on pre-Courant algebroids and curved pre-$L_\infty$-algebras are related under the map $\mathcal{M}$. This leads to a relationship between Nijenhuis operators and also enable us to see some structures on Lie algebroids as Nijenhuis forms on $L_\infty$-algebras.

Consider a skew-symmetric endomorphism $J:A \oplus A^* \to A \oplus A^*$ given as in (\ref{skewJ}):
\begin{equation*}
J= \left(
\begin{array}{cc}
 N & \pi^{\sharp} \\
\omega^{\flat} & -N^*
\end{array}
\right).
\end{equation*}
Recall that $J$ is identified with $\pi+N+\omega \in \Gamma(\wedge^2(A \oplus A^*))$ and that $J(X+\alpha)=\bb{X+\alpha}{\pi+N+\omega}$.

\

Let us define the \emph{extensions} $\underline{N}$ and $\underline{\omega}$ of the tensors $N$ and $\omega$, respectively, by setting, for all functions $f \in C^\infty(M)$ and homogeneous elements $P=P_1\wedge\ldots\wedge P_p \in \Gamma(\wedge^p A)$ and $Q=Q_1\wedge\ldots\wedge Q_q \in \Gamma(\wedge^q A)$,
\begin{equation*}\label{extension_N}
\left\{
  \begin{array}{l}
    \underline{N}(f)=0\\
    \underline{N}(P)=%\sum_{i=1}^p P_1\wedge\ldots\wedge P_{i-1}\wedge N(P_i)\wedge P_{i+1}\wedge\ldots\wedge P_p=
    \sum_{i=1}^p (-1)^{i-1}\,N(P_i)\wedge \widehat{P_i},
  \end{array}
\right.
\end{equation*}
and
\begin{equation*}\label{extension_omega}
\left\{
  \begin{array}{l}
    \underline{\omega}(P,f)=\underline{\omega}(f,P)=0\\
    \underline{\omega}(P,Q)=\sum_{i=1}^p\sum_{j=1}^q (-1)^{p+i+j-1}\omega(P_i, Q_j)\widehat{P_i}\wedge \widehat{Q_j},
  \end{array}
\right.
\end{equation*}
where $\widehat{P_i}=P_1\wedge\ldots\wedge P_{i-1}\wedge P_{i+1}\wedge\ldots\wedge P_p$ and $\widehat{Q_j}=Q_1\wedge\ldots\wedge Q_{j-1}\wedge Q_{j+1}\wedge\ldots\wedge Q_q$.

\begin{lem}\label{lem_extension_properties}
The extensions $\underline{N}$ and $\underline{\omega}$ are multiderivation symmetric vector-valued $1$-form and $2$-form, respectively, i.e.,
\begin{enumerate}
  \item
   $\underline{N}(P\wedge Q)=\underline{N}(P)\wedge Q + (-1)^{pq}\,\underline{N}(Q)\wedge P,$
   \item
   $\underline{\omega}(P,Q)=(-1)^{pq}\,\underline{\omega}(Q,P)$,
   \item
       $ \underline{\omega}(P,Q\wedge R)= \underline{\omega}(P,Q)\wedge R + (-1)^{qr}\,\underline{\omega}(P,R)\wedge Q$,
  \end{enumerate}
  for all $P \in \Gamma(\wedge^p A)$, $Q \in \Gamma(\wedge^q A)$ and $R \in \Gamma(\wedge^\bullet A)$.
   \end{lem}
\begin{proof}
Let us consider homogeneous elements $P=P_1\wedge\ldots\wedge P_p \in \Gamma(\wedge^p A)$, $Q=Q_1\wedge\ldots\wedge Q_q \in \Gamma(\wedge^q A)$ and $R=R_1\wedge\ldots\wedge R_r \in \Gamma(\wedge^r A)$. Then,

  \begin{enumerate}
    \item \begin{align*}
          \underline{N}(P\wedge Q)&= \sum_{i=1}^p (-1)^{i-1}\,N(P_i)\wedge \widehat{P_i}\wedge Q +  \sum_{j=1}^q (-1)^{p+j-1}\,N(Q_j)\wedge P\wedge \widehat{Q_j}\\
         % &= \left(\sum_{i=1}^p (-1)^{i-1}\,N(P_i)\wedge \widehat{P_i}\right)\wedge Q + (-1)^p P\wedge \left(\sum_{j=1}^q (-1)^{p+j-1}\,N(Q_j)\wedge \widehat{Q_j}\right)\\
          &= \underline{N}(P)\wedge Q + \left(\sum_{j=1}^q (-1)^{pq+j-1}\,N(Q_j)\wedge \widehat{Q_j}\right) \wedge P\\
          &= \underline{N}(P)\wedge Q + (-1)^{pq}\,\underline{N}(Q)\wedge P.
        \end{align*}
    \item  \begin{align*}
          \underline{\omega}(P,Q)&=\sum_{i=1}^p\sum_{j=1}^q (-1)^{p+i+j-1}\omega(P_i, Q_j)\widehat{P_i}\wedge \widehat{Q_j}\\
          &=\sum_{i=1}^p\sum_{j=1}^q (-1)^{p+i+j+(p-1)(q-1)}\omega(Q_j, P_i)\,\widehat{Q_j}\wedge \widehat{P_i}\\
          &=(-1)^{pq}\sum_{i=1}^p\sum_{j=1}^q (-1)^{q+i+j-1}\omega(Q_j, P_i)\,\widehat{Q_j}\wedge \widehat{P_i}\\
          &=(-1)^{pq}\,\underline{\omega}(Q,P).
        \end{align*}
       \item \begin{align*}
          \underline{\omega}(P,Q\wedge R)&= \sum_{i=1}^p\left(\sum_{j=1}^q (-1)^{p+i+j-1}\omega(P_i, Q_j)\,\widehat{P_i}\wedge \widehat{Q_j}\wedge R \right.\\
           & \quad \quad \left. + \sum_{k=1}^r (-1)^{p+i+q+k-1}\omega(P_i, R_k)\,\widehat{P_i}\wedge Q \wedge \widehat{R_k}\right)\\
          &= \left(\sum_{i=1}^p\sum_{j=1}^q (-1)^{p+i+j-1}\omega(P_i, Q_j)\,\widehat{P_i}\wedge \widehat{Q_j}\right)\wedge R \\
          &\quad\quad + \sum_{i=1}^p\sum_{k=1}^r (-1)^{p+i+q+k-1+q(r-1)}\omega(P_i, R_k)\, \widehat{P_i}\wedge \widehat{R_k} \wedge Q\\
          &= \underline{\omega}(P,Q)\wedge R + (-1)^{qr}\,\underline{\omega}(P,R)\wedge Q.
        \end{align*}
  \end{enumerate}
\end{proof}

 If one uses the graded Poisson bracket (big bracket) on the graded symplectic manifold $T^*[2]A[1]$, the evaluation of  $\underline{N}$ and  $\underline{\omega}$ on sections of $\wedge^\bullet A$ is much simpler to handle, as it is shown in the next lemma.
%\end{lem}
\begin{lem}\label{lem_j_in_bb}
For all $P, Q \in \Gamma(\wedge^\bullet A)$, we have:
  \begin{enumerate}
    \item $\underline{N}(P)=\bb{-N}{P}$;
    \item $\underline{\omega}(P,Q)=\bb{\bb{-\omega}{P}}{Q}$.
  \end{enumerate}
\begin{proof}
  All the operators are derivations on each entry, so we only have to check the identities for sections of $\Gamma(A)$. But in this case the identities are obvious because the extensions $\underline{N}$ and $\underline{\omega}$ coincide with the original tensors $N$ and $\omega$.
\end{proof}
\end{lem}
\

\begin{rem}
  In general, a tensor $\mathcal{V} \in \Gamma(\wedge^k A^* \otimes \wedge^l A)$ can be extended as a graded symmetric multiderivation $\underline{\mathcal{V}}\in S^\bullet L^*\otimes L$ of degree $k+l-2$, by setting
  \begin{multline*}
    \underline{\mathcal{V}}(P^1,\ldots,P^k)=(-1)^{\frac{k(k+1)}{2}+k(n_1+\ldots+n_k)-\sum_{j=1}^{k-1}j\,n_j}\\ \sum_{a_1=1}^{n_1}\ldots\sum_{a_k=1}^{n_k}(-1)^{a_1+\ldots+a_k}\, \mathcal{V}(P^1_{a_1}, P^2_{a_2}, \ldots, P^k_{a_k})\wedge \widehat{P^1_{a_1}}\wedge\ldots\wedge \widehat{P^k_{a_k}},
  \end{multline*}
  for all homogeneous elements $P^i=P^i_{1}\wedge \ldots \wedge P^i_{n_i} \in \Gamma(\wedge^{n_i} A)$ and where we used the notation $\widehat{P^i_{a_i}}=P^i_{1}\wedge \ldots \wedge P^i_{a_{i}-1}\wedge P^i_{a_{i}+1}\wedge \ldots \wedge P^i_{n_i}$, $i=1,\ldots,k$. Using derived brackets, the extension $\underline{\mathcal{V}}$ is simply defined by
  \begin{equation} \label{general_extension}
  \underline{\mathcal{V}}(P^1,\ldots,P^k)=(-1)^{kl-\frac{k(k-1)}{2}}\bb{\ldots\bb{\bb{\mathcal{V}}{P^1}}{P^2}}{\ldots,P^k}.\end{equation}
  %Define, using bb, the extension by derivation of $\underline{V}$ for any $V\in \Gamma(\wedge^k A^*\otimes\wedge^l A)$ and deduce an explicit formula (without bb and with signs).
\end{rem}

\

Let us now consider the map

\begin{equation*}\label{functor_upsilon}
\Upsilon: \xymatrixcolsep{3pc}\xymatrix@R=1pc{
    \Gamma(\wedge^2(A \oplus A^*)) \ar@{->}[r] &  S^{\bullet}(L^*)\otimes L\\
    J=\pi+N+\omega \ar@{|->}[r]&  \mathcal{j}=\mathcal{j}_0+\mathcal{j}_1+\mathcal{j}_2
    }
\end{equation*}

where
\begin{itemize}
\item  $\Upsilon(\pi)=\mathcal{j}_0 \in L_0 \subset S^0(L^*)\otimes L$ is defined by
\begin{equation}\label{definition_j0}
\mathcal{j}_0=-\pi;
\end{equation}

\item  $\Upsilon(N)=\mathcal{j}_1 \in S^1(L^*)\otimes L$ is defined by
\begin{equation}\label{definition_j1}
\mathcal{j}_1(P)=\underline{N}(P);
\end{equation}
\item  $\Upsilon(\omega)=\mathcal{j}_2 \in S^2(L^*)\otimes L$ is defined by
    \begin{equation}\label{definition_j2}
        \mathcal{j}_2(P,Q)=\underline{\omega}(P,Q),
    \end{equation}
\end{itemize}
for all $P,Q \in \Gamma(\wedge^{\bullet} A)$.

\

By Lemma \ref{lem_j_in_bb}, we can rewrite the expressions defining the map $\Upsilon$ using the big bracket as follows:
$$\left\{
  \begin{array}{ll}
    \mathcal{j}_0=-\pi\\
    \mathcal{j}_1(P)=\bb{-N}{P} \\
    \mathcal{j}_2(P,Q)=\bb{\bb{-\omega}{P}}{Q}.
  \end{array}
\right.$$

\

\begin{lem}
  The map $\Upsilon(J)=\mathcal{j}:S^{\bullet}(L)\rightarrow L$ is a graded symmetric linear map of degree zero.
\end{lem}
\begin{proof}
The map $\mathcal{j}$ is of degree zero because $\mathcal{j}_0 \in L_0$ and $\mathcal{j}_1$ and $\mathcal{j}_2$ are both maps of degree zero. To check this we use the same procedure as in the proof of Theorem \ref{prop_M(theta)_is_pre_L_infty}. For example, in the case of $\mathcal{j}_2$, if $P \in L_p$ and $Q \in L_q$, then $\mathcal{j}_2(P,Q)=\bb{\bb{-\omega}{P}}{Q}$ has bidegree
$$((0,2)+(p+2,0)+(-1,-1))+(q+2,0)+(-1,-1)=(p+q+2,0),$$ which means that $\mathcal{j}_2(P,Q)\in L_{p+q}$ and so  $\mathcal{j}_2$ is a map of degree zero.
  \end{proof}

Having shown in Theorem \ref{prop_M(theta)_is_pre_L_infty} that the map $\mathcal{M}$ is invertible, it seems natural to ask if $\Upsilon$ also admits an inverse. The answer is yes. Indeed,
given three maps $\mathcal{j}_i:S^{i}(L)\rightarrow L$, $i=0,1,2$, of degree zero and such that $\imath_f\mathcal{j_i}=0$, for all $f \in L_{-2}=C^{\infty}(M)$, i.e., $\mathcal{j_i}$ is $C^\infty(M)$-multilinear, we can define $J \in \Gamma(\wedge^2(A \oplus A^*))$ such that $\Upsilon(J)=\mathcal{j}_0+\mathcal{j}_1+\mathcal{j}_2$. The proof is analogous to the proof of converse part of Theorem \ref{prop_M(theta)_is_pre_L_infty}.

\

Recall that $\Theta \in \mathcal{F}^3$ can be deformed by a skew-symmetric endomorphism $J$ of $A\oplus A^*$, yielding $\Theta_J$ (see Section \ref{section_1}). Also, a curved pre-$L_\infty$-algebra can be deformed by a degree zero symmetric vector-valued form $\mathcal{n}$ yielding the curved pre-$L_\infty$-algebra $\mathcal{l}_{\mathcal{n}}=[\mathcal{n}, \mathcal{l}]$ (see Section \ref{section_3}).

A way to confirm that the maps $\mathcal{M}$ and $\Upsilon$ are a natural way to embed skew-symmetric endomorphisms of split pre-Courant algebroids into vector-valued forms on curved pre-$L_{\infty}$-structures is by checking that the following diagram is commutative:

\begin{equation}\label{diagram}
\vcenter{
\xymatrixcolsep{6pc}\xymatrix@R=3.5pc{
    \Theta \ar@{->}[r]^(.45){\mathcal{M}}\ar@{}[dr]|{\rotatebox{270}{\scalebox{2}{$\circlearrowleft$}}}
     \ar@{->}[d]_{\txt{\scriptsize deformation\\ \scriptsize by $J$}}
     &  l=\mathcal{M}(\Theta)
      \ar@{->}[d]^{\txt{\scriptsize deformation\\ \scriptsize by $\Upsilon(J)=\mathcal{j}$}}
      \\
    \Theta_J \ar@{->}[r]^{\mathcal{M}}&  l_{\mathcal{j}}
    }
    }
\end{equation}
This is the purpose of the next theorem.

\begin{thm}\label{thm_diagrama_comutativo}
Let $(A\oplus A^*, \Theta)$ be a pre-Courant algebroid and $J:A \oplus A^* \to A \oplus A^*$ a skew-symmetric endomorphism.
The diagram (\ref{diagram}) is commutative, which means that
  $$\mathcal{M}(\Theta_J)=(\mathcal{M}(\Theta))_{\Upsilon(J)},$$
where $\mathcal{M}$ and $\Upsilon$ are defined by Equations (\ref{definition_l0})-(\ref{definition_l3}) and (\ref{definition_j0})-(\ref{definition_j2}), respectively.
\end{thm}
\begin{proof}
Let us take $\Theta=\psi + \gamma + \mu + \phi$ and $J=\pi+N+\omega$ (see equations (\ref{Theta}) and (\ref{skewJ})).
Then, $$\Theta_J=\bb{J}{\Theta}=\bb{\pi+N+\omega}{\psi + \gamma + \mu + \phi}$$
can be decomposed as follows:
$$\left\{
  \begin{array}{ll}
    (\Theta_J)_{(3,0)}=\bb{\pi}{\gamma}+\bb{N}{\psi} \\
    (\Theta_J)_{(2,1)}=\bb{\pi}{\mu}+\bb{N}{\gamma}+\bb{\omega}{\psi}\\
    (\Theta_J)_{(1,2)}=\bb{\pi}{\phi}+\bb{N}{\mu}+\bb{\omega}{\gamma}\\
    (\Theta_J)_{(0,3)}=\bb{N}{\phi}+\bb{\omega}{\mu}.
  \end{array}
\right.$$
Now, Equations (\ref{definition_l0})-(\ref{definition_l3}) define explicitly the brackets forming $\mathcal{M}(\Theta_J)$:
\begin{equation}\label{equations_M(Theta_J)}
\left\{
\begin{array}{l}
\left(\mathcal{M}(\Theta_J)\right)_0=\bb{\pi}{\gamma}+\bb{N}{\psi}\\
\left(\mathcal{M}(\Theta_J)\right)_1(P)=\bb{\bb{\pi}{\mu}+\bb{N}{\gamma}+\bb{\omega}{\psi}}{P}\\
\left(\mathcal{M}(\Theta_J)\right)_2(P,Q)=\bb{\bb{\bb{\pi}{\phi}+\bb{N}{\mu}+\bb{\omega}{\gamma}}{P}}{Q}\\
\left(\mathcal{M}(\Theta_J)\right)_3(P,Q,R)=\bb{\bb{\bb{\bb{N}{\phi}+\bb{\omega}{\mu}}P}Q}R
\end{array}\right.
\end{equation}
for all $P, Q, R \in \Gamma(\wedge^{\bullet} A)$.

On the other hand, $\mathcal{M}(\Theta)=l=l_0+l_1+l_2+l_3$ is defined by Equations (\ref{definition_l0})-(\ref{definition_l3}) while $\Upsilon(J)=\mathcal{j}=\mathcal{j}_0+\mathcal{j}_1+\mathcal{j}_2$ is defined by Equations (\ref{definition_j0})-(\ref{definition_j2}). Thus, the curved pre-$L_\infty$-structure
$$(\mathcal{M}(\Theta))_{\Upsilon(J)}=[\mathcal{j}, l]=[\mathcal{j}_0+\mathcal{j}_1+\mathcal{j}_2, l_0+l_1+l_2+l_3]$$
can be decomposed in four terms $[\mathcal{j}, l]_i \in S^i(L^*)\otimes L, i=0,1,2,3$, as follows:
\begin{equation}\label{equations_M(Theta)_j}
\left\{
  \begin{array}{l}
    [\mathcal{j}, l]_0=[\mathcal{j}_0, l_1] + [\mathcal{j}_1, l_0]\\
    {[\mathcal{j}, l]_1}=[\mathcal{j}_0, l_2] + [\mathcal{j}_1, l_1]+ [\mathcal{j}_2, l_0]\\
    {[\mathcal{j}, l]_2}=[\mathcal{j}_0, l_3] + [\mathcal{j}_1, l_2]+ [\mathcal{j}_2, l_1]\\
    {[\mathcal{j}, l]_3}=[\mathcal{j}_1, l_3] + [\mathcal{j}_2, l_2].
  \end{array}
\right.
\end{equation}
We need to prove that the curved pre-$L_\infty$-structures defined by Equations (\ref{equations_M(Theta_J)}) and (\ref{equations_M(Theta)_j}) coincide. For the $0$-brackets, we have
\begin{align*}
  [\mathcal{j}_0, l_1] + [\mathcal{j}_1, l_0] &= l_1(\mathcal{j}_0) -  \mathcal{j}_1(l_0)\\
   &=\bb{\gamma}{-\pi} - \bb{-N}{\psi}\\
   &=\bb{\pi}{\gamma} + \bb{N}{\psi},
\end{align*}
where we used Equations (\ref{definition_RNbracket}), (\ref{definition_l0}), (\ref{definition_l1}), (\ref{definition_j0}) and (\ref{definition_j1}).
The $1$-brackets require more computations, but still straightforward. Besides the equations used for the $0$-brackets we also use Equations (\ref{definition_l2}) and (\ref{definition_j2}) to carry out the computations, for any $P \in L$:
\begin{align*}
  \left([\mathcal{j}_0, l_2] + [\mathcal{j}_1, l_1]+ [\mathcal{j}_2, l_0]\right)(P)
   &=l_2(\mathcal{j}_0, P) +l_1(\mathcal{j}_1(P)) - \mathcal{j}_1(l_1(P)) - \mathcal{j}_2(l_0, P)\\
   &=l_2(-\pi, P) +l_1(\bb{-N}{P}) - \mathcal{j}_1(\bb{\gamma}{P}) - \mathcal{j}_2(\psi, P)\\
   &=\bb{\bb{\mu}{-\pi}}{P}+ \bb{\gamma}{\bb{-N}{P}} - \bb{-N}{\bb{\gamma}{P}}\\
   & \quad \quad  - \bb{\bb{-\omega}{\psi}}{P}\\
   &=\bb{\bb{\pi}{\mu}}{P} - \bb{\bb{\gamma}{N}}{P} + \bb{\bb{\omega}{\psi}}{P}\\
   &=\bb{\bb{\pi}{\mu}+\bb{N}{\gamma}+ \bb{\omega}{\psi}}{P}.
\end{align*}
 For $2$-brackets and $3$-brackets, computations are more laborious (because they implicate more terms) but are similar.
\end{proof}

In the next theorem we shall use Theorem \ref{thm_diagrama_comutativo} in order to relate some classes of Nijenhuis endomorphisms on $(A\oplus A^*, \Theta)$ with Nijenhuis vector-valued forms on $(L, l)$.

\begin{thm}\label{thm_Nijenhuis_caso_J2=id}
Let $(A\oplus A^*, \Theta)$ be a pre-Courant algebroid and $J$ be a skew-symmetric endomorphism of $A \oplus A^*$ such that $J^2=\lambda \, {\mathrm{id}}_{A\oplus A^*}$, for some $\lambda \in \mathbb{R}$.

Then, $J$ is a Nijenhuis morphism on $(A\oplus A^*, \Theta)$ if and only if $\mathcal{j}=\Upsilon(J)$ is a Nijenhuis vector-valued form on the curved pre-$L_\infty$-structure $l=\mathcal{M}(\Theta)$ with square $\mathcal{k}=-\lambda\, \mathcal{E}$,
where $\mathcal{E}$ is the Euler map defined by (\ref{eqn_EulerMap}).
\end{thm}

Let us prove a calculatory lemma before proving Theorem \ref{thm_Nijenhuis_caso_J2=id}

\begin{lem}\label{lem_[k,li]}
  Let $\mathcal{k}$ be the vector-valued form on $L$, of degree zero, given by $\mathcal{k}=-\lambda \mathcal{E}$, for some $\lambda \in \mathbb{R}$.
  \begin{enumerate}
  \item
  If $l_i, i=0,1,2,3,$ are graded symmetric brackets of degree $1$ on $L$, then
  $$[\mathcal{k}, l_i]=\lambda\, l_i,\quad i=0,1,2,3.$$
  \item
  If $\mathcal{j}_i, i=0,1,2,$ are graded symmetric vector-valued forms of degree zero on $L$, then
  $$[\mathcal{j}_i,\mathcal{k}]=0,\quad i=0,1,2.$$
  \end{enumerate}
\end{lem}
\begin{proof}
  \begin{enumerate}
  \item
  The proof is done directly. We present here, as an example, the computations for $i=2$. For all $P \in  L_{p}$ and $Q \in  L_{q}$, we have
  \begin{align*}
    [\mathcal{k}, l_2](P,Q) &= (\imath_{\mathcal{k}} l_2 - \imath_{l_2}\mathcal{k})(P,Q) \\
    &= l_2(\mathcal{k}(P), Q) + (-1)^{pq}\, l_2(\mathcal{k}(Q), P) - \mathcal{k}(l_2(P,Q))\\
    &= -\lambda\, p\, l_2(P, Q) - \lambda\, q\,(-1)^{pq}\,  l_2(Q, P) + \lambda(p+q+1)\,l_2(P,Q)\\
    &= \lambda\, l_2(P,Q).
  \end{align*}
  \item
  Analogous to i).
  \end{enumerate}
\end{proof}

Let us now prove Theorem \ref{thm_Nijenhuis_caso_J2=id}

\begin{proof}[Proof of Theorem \ref{thm_Nijenhuis_caso_J2=id}]
The statement of Lemma \ref{lem_[k,li]} ii) ensures that $\mathcal{j}$ is a Nijenhuis vector-valued form on $l=\mathcal{M}(\Theta)$ with square $\mathcal{k}$ if and only if
\begin{equation}\label{eq_aux_1}
[\mathcal{j}, [\mathcal{j}, l]]=[\mathcal{k},l].
\end{equation}
Using Theorem \ref{thm_diagrama_comutativo} twice, for the l.h.s. of Equation (\ref{eq_aux_1}), we have :  $$[\mathcal{j}, [\mathcal{j}, l]]=\left(l_\mathcal{j}\right)_\mathcal{j}=\mathcal{M}\left(\left(\Theta_J\right)_J\right).$$
Furthermore, by Lemma \ref{lem_[k,li]} i) we know that $$[\mathcal{k},l]=\lambda\, l=\lambda\, \mathcal{M}(\Theta)=\mathcal{M}(\lambda\,\Theta).$$
Thus, Equation (\ref{eq_aux_1}) is equivalent to $$\mathcal{M}(\left(\Theta_J\right)_J)=\mathcal{M}(\lambda\,\Theta),$$
which is equivalent to $\left(\Theta_J\right)_J=\lambda\,\Theta$, because the map $\mathcal{M}$ is injective (this is part of the proof of Theorem \ref{prop_M(theta)_is_pre_L_infty}). But, using Equation (\ref{supergeometric_torsion}), this is equivalent to $J$ being a Nijenhuis endomorphism on $(A\oplus A^*, \Theta)$.
\end{proof}

\begin{rem}
  In Theorem \ref{thm_Nijenhuis_caso_J2=id}, if $\Theta$ is a Courant algebroid structure then $\Theta_J$ is also a Courant algebroid structure and both $l$ and $l_{\mathcal{j}}$ are curved $L_\infty$-structures.
\end{rem}

%
%%%%%%%%%%%%%%%%%%%%%%%%%%%%%%%%%%%%%%%%%%%%%%%%%%%%%%%%%%%%%%%%%%%%%%%%%%%%%%%%%%%%%%%%%%%%%%%%%%%%%%%%%%%%%%%%%%%%%%

As it was mentioned in Section \ref{section_1}, $(A \oplus A^*, \mu)$ is a Courant algebroid if and only if $(A, \mu)$ is a Lie algebroid. Moreover, if $\Theta=\mu + \phi$, from (\ref{eq_(theta,theta)=0}) we have that
%\begin{center}
$(A \oplus A^*, \mu + \phi)$ is a Courant algebroid if and only if $(A, \mu)$ is a Lie algebroid and $\{\mu, \phi \}=0$.
% \end{center}
 The condition $\{\mu, \phi \}=0$ means that $\phi$ is a closed $3$-form on the Lie algebroid $(A, \mu)$, and we write $\mathrm{d} \phi=0$.

Poisson quasi-Nijenhuis structures
 with background were intoduced in \cite{A10} as quadruples $(\pi, N, \varphi, H)$ formed by a bivector $\pi$, a $(1,1)$-tensor $N$ and two closed $3$-forms $\varphi$ and $H$ on a Lie algebroid $(A, \mu)$, satisfying some conditions. In the case where $\varphi$ is exact, $\varphi=\mathrm{d} \omega$, we have an exact Poisson quasi-Nijenhuis structure with background  \cite{{AntunesCosta2013}}. Denoting by $C(\pi, N)$ the Magri-Morosi concomitant of $\pi$ and $N$ and by $ {\text{\Fontlukas T}}N $ the Nijenhuis torsion of $N$ on $(A, \mu)$, the definition goes as follows:

\begin{defn}\cite{{AntunesCosta2013}} \label{def_ExactPQNB}
An {\em exact Poisson quasi-Nijenhuis structure
 with background} on a Lie algebroid $(A, \mu)$ is a quadruple $(\pi, N, \omega, H)$, where $\pi$ is a bivector, $\omega$ is a $2$-form, $N$ is a $(1,1)$-tensor and $H$ is a closed
 $3$-form such that $N \circ \pi^\#= \pi^\# \circ N^*$, $\omega^\flat \circ N= N^* \circ \omega^\flat$  and
 \begin{enumerate}
 \item[(i)] $\pi$ is Poisson,
 \item[(ii)] $C(\pi, N)(\alpha, \beta)= 2 H(\pi^\#(\alpha), \pi^\#(\beta),.)$, for all $\alpha, \beta \in \Gamma(A^*)$,
 \item[(iii)] $ {\text{\Fontlukas T}}N (X,Y)= \pi^\# ( H(NX,Y,.)+ H(X,NY,.)+ \mathrm{d} \omega(X,Y,.))$, for all $ X,Y \in \Gamma(A)$,
 \item[(iv)]
 $i_N \mathrm{d} \omega- \mathrm{d} \omega_N- {\mathcal H} + \lambda H =0 $, for some $\lambda \in {\mathbb R}$,
 \end{enumerate}
  with $\omega_N(X,Y):=\omega(NX,Y)$ and
  $\mathcal H (X,Y,Z):= \circlearrowleft_{X,Y,Z} H(NX,NY,Z)$, for all $X,Y,Z \in \Gamma(A)$, where $\circlearrowleft_{X,Y,Z}$
means sum after circular permutation on $X$, $Y$ and $Z$.
\end{defn}

In \cite{AntunesCosta2013} we proved that, given  a closed $3$-form $H \in \Gamma(\wedge^3 A^*)$ on a Lie algebroid $(A, \mu)$ and a skew-symmetric endomorphism $J$ of $A \oplus A^*$, \begin{equation*}
J= \left(
\begin{array}{cc}
 N & \pi^{\sharp} \\
\omega^{\flat} & -N^*
\end{array}
\right),
\end{equation*} such that $J^2=\lambda \, {\textrm{id}}_{A\oplus A^*}$, for some $\lambda \in \mathbb{R}$, then
%\begin{center}
\emph{
$J$ is a Nijenhuis morphism on the Courant algebroid $(A \oplus A^*, \mu + H)$  if and only if the quadruple $(\pi, N,\omega,H)$ is an exact
   Poisson quasi-Nijenhuis structure with background on $(A, \mu)$.}
 %   \end{center}

From Theorem \ref{thm_Nijenhuis_caso_J2=id}, we get that each exact Poisson quasi-Nijenhuis structure with background on a Lie algebroid $(A, \mu)$ can be seen as a Nijenhuis vector-valued form with respect to a curved $L_\infty$-algebra structure on $\Gamma(\wedge^\bullet A)[2]$. More precisely, we have:

\begin{cor} \label{cor_ePqNb}
Let $(A,\mu)$ be a Lie algebroid, $\pi$ a bivector, $\omega$ a $2$-form, $N$ a $(1,1)-$tensor  and $\phi$ a closed $3$-form on $(A,\mu)$. Assume that $N \circ \pi^\#= \pi^\# \circ N^*$, $\omega^\flat \circ N= N^* \circ \, \omega^\flat$  and $N^2= \lambda \, \mathrm{id}_A $, for some $\lambda \in {\mathbb R}$. Then, the quadruple $(\pi, N, \omega, \phi)$ is an exact Poisson quasi-Nijenhuis structure with background on $(A, \mu)$ if and only if
 $\mathcal{j}=- \pi + \underline{N} + \underline{\omega}$ is a Nijenhuis vector-valued form on the $L_\infty$-algebra $(\Gamma(\wedge^{\bullet}A)[2], l_2+l_3)$, with square $\mathcal{k}=-\lambda\, \mathcal{E}$.
\end{cor}

In \cite{ALC18} a one-to-one correspondence between an exact Poisson quasi-Nijenhuis structure with background and a co-boundary Nijenhuis\footnote{If we remove condition  $[\mathcal{n},\mathcal{k}]=0$ in Definition \ref{def_Nijenhuis}, $\mathcal{n}$ is called a co-boundary Nijenhuis form.} vector-valued form is established. Indeed, Theorem 4.4 in \cite{ALC18} establishes that
\emph{ $(\pi, N,-\omega, \phi )$ is an exact Poisson quasi-Nijenhuis structure with background on a Lie algebroid $(A, \mu)$ if and only if ${\mathcal N}=\pi + \underline{N} + \underline{\omega}$ is a co-boundary Nijenhuis vector-valued form on the $L_\infty$-algebra $(\Gamma(\wedge^{\bullet}A)[2], l_2+ \underline{\phi}=l_2 - l_3)$}\footnote{The extension $\underline{\phi}$ of the $3$-form $\phi$ is given by (\ref{general_extension}). More precisely, for all $P,Q,R \in \Gamma(\wedge^\bullet A)$, $\underline{\phi}(P,Q,R)= \{ \{\{- \phi, P\}, Q \}, R\}=-l_3(P,Q,R)$ and Lemma \ref{lema_da tese} yields $\underline{\phi}=- l_3$.}, with square $\underline{N^2}+ [\underline{\omega}, \pi]$.

The approach in \cite{ALC18} is different from the one considered in the current paper since, contrary to what happens in Corollary \ref{cor_ePqNb}, $\Upsilon(\pi+ N - \omega)=\pi + \underline{N} - \underline{\omega}\neq {\mathcal N}$.

When, in Definition \ref{def_ExactPQNB}, $\omega=0$ and $H=0$, the pair $(\pi, N)$ is a \emph{Poisson-Nijenhuis structure} on the Lie algebroid $(A, \mu)$.
In the case where $N^2= \lambda \, {\mathrm{id}}_A$, for some $\lambda \in \mathbb R$, Theorem \ref{thm_Nijenhuis_caso_J2=id} gives the following characterization of these Poisson-Nijenhuis structures in the setting of $L_\infty$-algebras (see \cite{AntunesCosta2013}).
\begin{cor}
Let $(A,\mu)$ be a Lie algebroid, $\pi$ a bivector and $N$ a $(1,1)$-tensor such that $N \circ \pi^\#= \pi^\# \circ N^*$  and $N^2= \lambda \,{\mathrm{id}}_A $, for some $\lambda \in {\mathbb R}$. Then, the pair $(\pi, N)$ is a Poisson-Nijenhuis structure on $(A, \mu)$ if and only if
 $\mathcal{j}=-\pi + \underline{N} $ is a Nijenhuis vector-valued form with respect to the $L_\infty$-algebra $(\Gamma(\wedge^{\bullet}A)[2], l_2)$, with square $\mathcal{k}=-\lambda\, \mathcal{E}$.
\end{cor}

Recall that an \emph{$\Omega N$ structure} on a Lie algebroid $(A,\mu)$ is a pair $(\omega, N)$, where $N$ is a Nijenhuis tensor, $\omega$  is a closed $2$-form such that $\omega^\flat \circ N= N^* \circ \omega^\flat$ and the $2$-form $\omega_N(\cdot, \cdot)=\omega(N \cdot, \cdot)$ is closed.

In \cite{AntunesCosta2013} we proved that, given  a closed $2$-form $\omega$ on a Lie algebroid $(A, \mu)$ and a skew-symmetric endomorphism $J_{\omega, N}$ of $A \oplus A^*$, \begin{equation*}
J_{\omega, N}= \left(
\begin{array}{cc}
 N & 0 \\
\omega^{\flat} & -N^*
\end{array}
\right),
\end{equation*} such that $J_{\omega, N}^2=\lambda \, {\textrm{id}}_{A\oplus A^*}$, for some $\lambda \in \mathbb{R}$, then
%\begin{center}
\emph{
$J_{\omega, N}$ is a Nijenhuis morphism on the Courant algebroid $(A \oplus A^*, \mu)$  if and only if  $(\omega,N)$ is an $\Omega N$
   structure on $(A, \mu)$.} So, from Theorem \ref{thm_Nijenhuis_caso_J2=id}, we get the following characterization of these $\Omega N$ structures in the setting of $L_\infty$-algebras.
   \begin{cor}
 Let $(A,\mu)$ be a Lie algebroid, $\omega$ a $2$-form and $N$ a $(1,1)$-tensor such that $\omega^\flat \circ N= N^* \circ \omega^\flat$  and $N^2= \lambda \, {\mathrm{id}}_A $, for some $\lambda \in {\mathbb R}$. Then, the pair $(\omega, N)$ is an $\Omega N$ structure on $(A, \mu)$ if and only if
 $\mathcal{j}= \underline{N}+ \underline{\omega}$ is a Nijenhuis vector-valued form with respect to the $L_\infty$-algebra $(\Gamma(\wedge^{\bullet}A)[2], l_2)$, with square $\mathcal{k}=-\lambda\, \mathcal{E}$.
   \end{cor}
 %   \end{center}

A \emph{$P \Omega$ structure} on a Lie algebroid $(A,\mu)$ is a pair $(\pi, \omega)$, where $\pi$ is a Poisson bivector and the $2$-forms $\omega$ and $\omega_N$ are closed,  with $N= \pi^\sharp \circ \omega^\flat$.

Using Theorem \ref{thm_Nijenhuis_caso_J2=id}, we may establish a relation between  a class of $P \Omega$ structures and $L_\infty$-algebras. For that purpose we need to recall the next two lemmas.

\begin{lem} \label{POmega}\cite{A10}
If $(\pi, \omega)$ is a $P \Omega$ structure on $(A,\mu)$, then $N= \pi^\sharp \circ \omega^\flat$ is a Nijenhuis tensor on $(A,\mu)$.
\end{lem}

\begin{lem} \label{Nijenhuis}\cite{YKS11}
Let $N$ be a $(1,1)$-tensor on a Lie algebroid $(A,\mu)$ such that $N^2= \lambda \,{\mathrm{id}}_A $, for some $\lambda \in {\mathbb R}$. Then $N$ is a Nijenhuis tensor on $(A,\mu)$ if and only if
$J_N: A \oplus A^* \to A \oplus A^*$ given by  %\begin{equation*}
$J_N= \left(
\begin{array}{cc}
 N & 0 \\
0 & -N^*
\end{array}
\right)$
%\end{equation*}
is a Nijenhuis morphism on the Courant algebroid $(A \oplus A^*, \mu)$.
\end{lem}

\begin{cor}
Let $(A,\mu)$ be a Lie algebroid, $\pi$ a Poisson bivector and $\omega$ a closed $2$-form such that $N:=\pi^\sharp \circ \omega^\flat$ satisfies  $N^2= \lambda \, {\mathrm{id}}_A $, for some $\lambda \in {\mathbb R}$. If the pair $(\pi, \omega)$ is a $P \Omega$ structure on $(A, \mu)$ then
 $\mathcal{j}= \underline{N}$ is a Nijenhuis vector-valued form with respect to the $L_\infty$-algebra $(\Gamma(\wedge^{\bullet}A)[2], l_2)$, with square $\mathcal{k}=-\lambda\, \mathcal{E}$.
  \end{cor}
  \begin{proof}
  It is a direct consequence of Theorem \ref{thm_Nijenhuis_caso_J2=id}, by application of Lemmas \ref{POmega} and \ref{Nijenhuis}.
  \end{proof}
%%%%%%%%%%%%%%%%%%%%%%%%%%%%%%%%%%%%%%%%%%%%%%%%%%%%%%%%%%%%%%%%%%%%%%%%%%%%%%%%%%%%%%%%%%%%%%%%%%%%%%%%%%%%%%%%%%%%%%%%
%%%%%%%%%%%%%%%%%%%%%%%%%%%%%%%%%%%%%%%%%%%%%%%%%%%%%%%%%%%%%%%%%%%%%%%%%%%%%%%%%%%%%%%%%%%%%%%%%%%%%%%%%%%%%%%%%%%%%%%%%
%%%%%%%%%%%%%%%%%%%%%%%%%%%%%%%%%%%%%%%%%%%%%%%%%%%%%%%%%%%%%%%%%%%%%%%%%%%%%%%%%%%%%%%%%%%%%%%%%%%%%%%%%%%%%%%%%%%%%%%%%

\section{Twisting by a bivector} \label{twisting by bivector}
The purpose of this section is to discuss the twisting of a Courant algebroid and of a curved $L_\infty$-algebra by a bivector.

Given a pre-Courant structure $\Theta$ on $A \oplus A^*$, the notion of \emph{twisting} $\Theta$ by a bivector $\pi \in \Gamma(\wedge ^2 A)$ was introduced in \cite{roy} as the canonical transformation
given by the flow of the Hamiltonian vector field $X_\pi:=\{\pi, \cdot \}$ associated to $\pi$:
\begin{equation*}
\emph{e}^{\pi}:= 1+ \{\pi, \cdot \} + \frac{1}{2 !}  \{\pi, \{\pi, \cdot \}\} + \frac{1}{3 !}  \{\pi, \{\pi, \{\pi, \cdot \}\}\} + \dots
\end{equation*}
When applied to $\Theta= \psi + \gamma + \mu + \phi$ yields
\begin{align*} \label{definition_e_pi_Theta}
\emph{e}^{\pi} \Theta = &  \psi + \{\pi, \gamma \}+  \frac{1}{2}  \{\pi, \{\pi, \mu \}\}+ \frac{1}{6}  \{\pi, \{\pi, \{\pi, \phi \}\}\}\\
& + \gamma +\{\pi, \mu \}+ \frac{1}{2}  \{\pi, \{\pi, \phi \}\} + \mu+ \{\pi, \phi \}  + \phi. \nonumber
\end{align*}
Since $X_\pi:=\{\pi, \cdot \}$ is of degree zero, $\emph{e}^{\pi} \Theta$ has degree $3$, that is to say $\emph{e}^{\pi} \Theta \in \mathcal{F}^{3}$ is a pre-Courant structure on $A \oplus A^*$. Moreover, we have the following:

\begin{prop}\cite{roy} \label{exponential}
If $\Theta$ is a Courant structure on $A \oplus A^*$ so is $e^{\pi} \Theta$.
%$(A \oplus A^*,  \langle\cdot,\cdot\rangle, e^{\pi} \Theta)$ is a Courant algebroid.
\end{prop}

Replacing $\pi \in \Gamma(\wedge ^2 A)$ by $\omega \in \Gamma(\wedge ^2 A^*)$ one has the twisting of $\Theta$ by $\omega$, $e^{\omega} \Theta$  \cite{roy}. Proposition \ref{exponential} also holds for $e^{\omega} \Theta$.

\begin{rem} The next Definition \ref{def_twisting_algebra} and Propositions \ref{Getzler_result} and \ref{prop_twist_curved} also hold, without any change, for curved pre-$L_\infty$-algebras but, for the sake of better reading, we shall not address them in the more general setting.
\end{rem}

Recall that a Maurer-Cartan element of a curved $L_\infty$-algebra $(\mathcal{L}, \sum_{i=0}^{3} \mathcal{l}_i)$ is a degree zero element $\pi \in \mathcal{L}_0$ such that
\begin{equation} \label{def_MCartan}
\mathcal{l}_0 -\mathcal{l}_1 (\pi)+ \frac{1}{2}\mathcal{l}_2( \pi,\pi) - \frac{1}{6}\mathcal{l}_3( \pi,\pi, \pi) =0.
\end{equation}

Let $(\mathcal{L}, \sum_{i\geq 0} \mathcal{l}_i)$ be a
 curved $L_\infty$-algebra and $\pi \in \mathcal{L}_0$, a degree zero element of $\mathcal{L}$. Let us define the operator

\begin{equation*}
\varepsilon^{\pi}:= 1 - [\pi, \cdot] + \frac{(-1)^2}{2 !}  [\pi, [\pi, \cdot ]] + \frac{(-1)^3}{3 !}  [\pi, [\pi, [\pi, \cdot ]]] + \dots
\end{equation*}
that, applied to $\mathcal{l}:=\sum_{i=0}^{3} \mathcal{l}_i$, yields:
\begin{align}
\varepsilon^{\pi}\mathcal{l} = & \underbrace{\mathcal{l}_0 - \mathcal{l}_1(\pi) +  \frac{1}{2}\mathcal{l}_2(\pi, \pi) - \frac{1}{6} \mathcal{l}_3(\pi, \pi, \pi)}_{(\varepsilon^\pi \mathcal{l})_0} \label{varepsilon_e_pi}\\
& + \underbrace{\mathcal{l}_1 - \mathcal{l}_2(\pi, \cdot)  +\frac{1}{2} \mathcal{l}_3(\pi, \pi, \cdot)}_{(\varepsilon^\pi \mathcal{l})_1}
 + \underbrace{\mathcal{l}_2 - \mathcal{l}_3(\pi, \cdot)}_{(\varepsilon^\pi \mathcal{l})_2} + \underbrace{\mathcal{l}_3}_{(\varepsilon^\pi \mathcal{l})_3}.  \nonumber
\end{align}
\begin{defn} \label{def_twisting_algebra}
The pair $(\mathcal{L}, \varepsilon^{\pi}\mathcal{l})$ is called the \emph{twisting} by $\pi$ of the curved $L_\infty$-algebra $(\mathcal{L}, \mathcal{l})$.
\end{defn}

When the curvature vanishes ($\mathcal{l}_0=0$) and so $(\mathcal{L}, \mathcal{l})$ is an $L_\infty$-algebra, in general, the term $(\varepsilon^\pi \mathcal{l})_0 \in \mathcal{L}_0$
need not to vanish.
The vanishing of $(\varepsilon^\pi \mathcal{l})_0$, which is equivalent to
  \begin{equation*}
 \mathcal{l}_1(\pi)- \frac{1}{2} \mathcal{l}_2(\pi,\pi) + \frac{1}{6} \mathcal{l}_3 (\pi,\pi, \pi)=0,
 \end{equation*}
 means that $\pi$ is a Maurer-Cartan element of the $L_\infty$-algebra $(\mathcal{L},\sum_{i=1}^{3} \mathcal{l}_i)$ (see (\ref{def_MCartan})). So, we recover a result from \cite{Getzler2009}:
 \begin{prop}  \label{Getzler_result}
The twisting by $\pi$ of the $L_\infty$-algebra $(\mathcal{L}, \mathcal{l})$ is an $L_\infty$-algebra provided that $\pi$ is a Maurer-Cartan element of $(\mathcal{L}, \mathcal{l})$.
\end{prop}

Let us see that when dealing with curved $L_\infty$-algebras, the condition of $\pi$ being a Maurer-Cartan element can be removed. Next proposition holds for any curved $L_\infty$-algebra, but we only consider the case where $\mathcal{l}_i=0$, for $i\geq 4$, which is the one we are interested in.

\begin{prop} \label{prop_twist_curved}
Let $(\mathcal{L}, \sum_{i=0}^{3} \mathcal{l}_i)$ be a curved $L_\infty$-algebra and $\pi$ a degree zero element of $\mathcal{L}$. Then, $(\mathcal{L}, \varepsilon^{\pi}\mathcal{l})$  is a curved $L_\infty$-algebra.
\end{prop}

\begin{proof}
We have to prove that, for all homogeneous $X, X_1, \cdots, X_5 \in \mathcal{L}$, and using the notation of (\ref{varepsilon_e_pi}):
\begin{align*}
 i) &(\varepsilon^\pi \mathcal{l})_1((\varepsilon^\pi \mathcal{l})_0)=0;\\
 ii)& (\varepsilon^\pi \mathcal{l})_2((\varepsilon^\pi \mathcal{l})_0, X)+  (\varepsilon^\pi \mathcal{l})_1((\varepsilon^\pi \mathcal{l})_1(X))=0; \\
 iii)&(\varepsilon^\pi \mathcal{l})_3((\varepsilon^\pi \mathcal{l})_0, X_1, X_2)+ (\varepsilon^\pi \mathcal{l})_1((\varepsilon^\pi \mathcal{l})_2 (X_1, X_2))\\
 & + (\varepsilon^\pi \mathcal{l})_2((\varepsilon^\pi \mathcal{l})_1(X_1), X_2) + (-1)^{x_{1}x_{2}}(\varepsilon^\pi \mathcal{l})_2((\varepsilon^\pi \mathcal{l})_1(X_2), X_1)=0;  \\
iv) &\sum_{\sigma \in Sh(1,2)}\epsilon(\sigma)(\varepsilon^\pi\mathcal{l})_3( (\varepsilon^\pi\mathcal{l})_1(X_{\sigma(1)}),X_{\sigma(2)}, X_{\sigma(3)})\\
 &+\sum_{\sigma \in Sh(2,1)}\epsilon(\sigma)(\varepsilon^\pi\mathcal{l})_2((\varepsilon^\pi\mathcal{l})_2(X_{\sigma(1)},X_{\sigma(2)}),X_{\sigma(3)}) \\
 &+(\varepsilon^\pi\mathcal{l})_1((\varepsilon^\pi\mathcal{l})_3( X_1, X_2, X_3))=0;\\
v)& \sum_{\sigma \in Sh(3,1)}\epsilon(\sigma)(\varepsilon^\pi \mathcal{l})_2((\varepsilon^\pi \mathcal{l})_3(X_{\sigma(1)},X_{\sigma(2)}, X_{\sigma(3)}), X_{\sigma(4)}) \\
& + \sum_{\sigma \in Sh(2,2)}\epsilon(\sigma)(\varepsilon^\pi \mathcal{l})_3((\varepsilon^\pi \mathcal{l})_2(X_{\sigma(1)},X_{\sigma(2)}), X_{\sigma(3)},X_{\sigma(4)})=0; \\
 vi)& \sum_{\sigma \in Sh(3,2)}\epsilon(\sigma)(\varepsilon^\pi \mathcal{l})_3((\varepsilon^\pi \mathcal{l})_3(X_{\sigma(1)},X_{\sigma(2)}, X_{\sigma(3)}), X_{\sigma(4)}, X_{\sigma(5)})=0,
\end{align*}
where $x_i$ stands for the degree of $X_i$.
For $i)$, we compute

\begin{align*}
(\varepsilon^\pi \mathcal{l})_1((\varepsilon^\pi \mathcal{l})_0)
& = (\varepsilon^\pi \mathcal{l})_1 (\mathcal{l}_0- \mathcal{l}_1(\pi)+ \frac{1}{2} \mathcal{l}_2(\pi,\pi) - \frac{1}{6} \mathcal{l}_3 (\pi,\pi, \pi))\\
&= \Big[\mathcal{l}_1(\mathcal{l}_0)\Big]  -\Big[ \mathcal{l}_1(\mathcal{l}_1(\pi))+ \mathcal{l}_2(\pi, \mathcal{l}_0)\Big] \\
& + \left[\frac{1}{2} \mathcal{l}_1(\mathcal{l}_2(\pi, \pi)) + \mathcal{l}_2(\pi, \mathcal{l}_1(\pi))+ \frac{1}{2} \mathcal{l}_3(\pi, \pi, \mathcal{l}_0)\right] \\
&- \left[\frac{1}{6}\mathcal{l}_1(\mathcal{l}_3(\pi, \pi, \pi)) + \frac{1}{2}\mathcal{l}_2(\pi, l_2(\pi, \pi))+ \frac{1}{2}\mathcal{l}_3(\pi, \pi, \mathcal{l}_1(\pi))\right] \\
& + \left[\frac{1}{6} \mathcal{l}_2(\pi, \mathcal{l}_3(\pi, \pi, \pi))+ \frac{1}{4} \mathcal{l}_3(\pi, \pi, \mathcal{l}_2( \pi, \pi))\right] - \left[ \frac{1}{12} \mathcal{l}_3(\pi, \pi, \mathcal{l}_3(\pi, \pi, \pi))\right].
\end{align*}

\noindent Since $(\mathcal{L}, \sum_{i=0}^{3} \mathcal{l}_i)$ is a curved $L_\infty$-algebra, each expression inside the brackets $[ \cdots ]$ is zero as a consequence of the generalized Jacobi identity (\ref{gen_Jacobi_id}) for $n=0, 1, \dots, 5$, with $X_i=\pi, \, i=1,\dots, 5$. Thus, $i)$ is proved.

For $ii)$, we have
\begin{align*}
&(\varepsilon^\pi \mathcal{l})_2((\varepsilon^\pi \mathcal{l})_0, X) +  (\varepsilon^\pi \mathcal{l})_1((\varepsilon^\pi \mathcal{l})_1(X))  =
 \Big[\mathcal{l}_2(\mathcal{l}_0, X) + \mathcal{l}_1(\mathcal{l}_1(X))\Big] \\
&-\Big[ \mathcal{l}_2(\mathcal{l}_1(\pi),X)+ \mathcal{l}_3(\pi, \mathcal{l}_0, X) + \mathcal{l}_1(\mathcal{l}_2(\pi, X))+ \mathcal{l}_2(\mathcal{l}_1(X), \pi) \Big]\\
&+ \left[ \frac{1}{2}\mathcal{l}_2(\mathcal{l}_2(\pi, \pi), X) + \mathcal{l}_3(\mathcal{l}_1(\pi),\pi, X)+ \frac{1}{2} \mathcal{l}_1(\mathcal{l}_3(\pi, \pi, X)) + \mathcal{l}_2(\mathcal{l}_2(\pi, X), \pi) + \frac{1}{2} \mathcal{l}_3(\mathcal{l}_1(X),\pi, \pi)\right]\\
&- \left[ \frac{1}{6} \mathcal{l}_2(\mathcal{l}_3(\pi, \pi, \pi), X) + \frac{1}{2} \mathcal{l}_3( \mathcal{l}_2(\pi, \pi),\pi, X) + \frac{1}{2}\mathcal{l}_2( \mathcal{l}_3(\pi, \pi, X), \pi) + \frac{1}{2} \mathcal{l}_3( \mathcal{l}_2(\pi, X),\pi,\pi)\right] \\
& +\left[ \frac{1}{6} \mathcal{l}_3(\mathcal{l}_3( \pi, \pi, \pi),\pi,X) + \frac{1}{4} \mathcal{l}_3(\mathcal{l}_3(\pi, \pi, X),\pi, \pi)\right].
\end{align*}
Again, using (\ref{gen_Jacobi_id})
we get that each expression inside the brackets $[ \cdots ]$ is zero, and $ii)$ is proved. The proofs of $iii)$, $iv)$, $v)$ and $vi)$ are similar.
\end{proof}

%%%%%%%%%%%%%%%%%%%%%%%%%%%%%%%%%%%%%%%%%%%%%%%%%%%%%%%%%%%%%%%%%%%%%%%%%%%%%%%%%%%%%%%%%%%%%%%%%%%%%%%%%%%%%%%%%%%%%%%%%%%%%%%%%%%%%%%%%%%%%%%%%%%%%%%%%%%%%%%%%%%%%%%%%%%%%%%%%%%%%%%%%%%%%%
%%%%%%%%%%%%%%%%%%%%%%%%%%%%%%%%%%%%%%%%%%%%%%%%%%%%%%%%%%%%%%%%%%%%%%%%%%%%%%%%%%%%%%%%%%%%%%%%%%%%%%%%%%%%%%%%%%%%%%%%%%%%%%%%%%%%%%%%%%%%%%%%%%%%%%%%%%%%%%%%%%%%%%%%%%%%%%%%%%%%%%%%%%%%%%%%%
%%%%%%%%%%%%%%%%%%%%%%%%%%%%%%%%%%%%%%%%%%%%%%%%%%%%%%%%%%%%%%%%%%%%%%%%%%%%%%%%%%%%%%%%%%%%%%%%%%%%%%%%%%%%%%%%%%%%%%%%%%%%%%%%%%%%%%%%%%%%%%%%%%%%%%%%%%%%%%%%%%%%%%%%%%%%%%%%%%%%%%%%%%%%%

Next we shall see that the map $\mathcal{M}$, given by Equations (\ref{definition_l0})-(\ref{definition_l3}), commutes with the operations of twisting by $\pi$.

Let $(A\oplus A^*, \Theta= \psi + \gamma + \mu + \phi)$ be a pre-Courant algebroid and $(L=\Gamma(\wedge^\bullet A)[2], l={\mathcal{M}}(\Theta))$ the corresponding curved pre-$L_\infty$-algebra constructed in Section \ref{section_4}. Let $\pi$ be a degree zero element of $L$, i.e., $\pi \in \Gamma(\wedge^2 A) =L_0$. Since $\Theta$ is a pre-Courant structure on $A \oplus A^*$, so it is $e^\pi \Theta$ and, according to Theorem \ref{prop_M(theta)_is_pre_L_infty}, ${\mathcal{M}}(e^\pi \Theta)$ is a curved pre-$L_\infty$-algebra.

\begin{prop} \label{prop_M_exp_pi_commute} We have,
\begin{equation*} \label{M_exp_pi_commute}
{\mathcal{M}}(e^\pi \Theta)= \varepsilon^\pi({\mathcal{M}}(\Theta))= \varepsilon^\pi l.
\end{equation*}
Equivalently, the next diagram is commutative:

\begin{equation*}\label{diagram4}
\vcenter{
\xymatrixcolsep{6pc}\xymatrix@R=3.5pc{
    \Theta \ar@{->}[r]^{\mathcal{M}}
     \ar@{->}[d]_{e^\pi}\ar@{}[dr]|{\rotatebox{270}{\scalebox{2}{$\circlearrowleft$}}}
     &  l
      \ar@{->}[d]^{\varepsilon^\pi}
      \\
    e^\pi \Theta \ar@{->}[r]^(.45){\mathcal{M}}&  \varepsilon^\pi l
    }
    }
\end{equation*}

\end{prop}

\begin{proof}
Applying ${\mathcal{M}}$ to

\begin{align*}
\emph{e}^{\pi} \Theta = &\underbrace{ \psi + \{\pi, \gamma \}+  \frac{1}{2}  \{\pi, \{\pi, \mu \}\}+ \frac{1}{6}  \{\pi, \{\pi, \{\pi, \phi \}\}\}}_{\textrm{bidegree}\, (3,0)}\\
 & + \underbrace{\gamma +\{\pi, \mu \}+ \frac{1}{2}  \{\pi, \{\pi, \phi \}\}}_{\textrm{bidegree}\, (2,1)} +  \underbrace{\mu+ \{\pi, \phi \}}_{\textrm{bidegree}\, (1,2)}+ \underbrace{\phi}_{\textrm{bidegree} \,(0,3)}\nonumber
\end{align*}
and using (\ref{definition_l0})-(\ref{definition_l3}), yields ${\mathcal{M}}(\emph{e}^{\pi} \Theta)=\sum_{i=0}^{3}({\mathcal{M}}(\emph{e}^{\pi} \Theta))_i$ with
$$\left\{
  \begin{array}{ll}
({\mathcal{M}}(\emph{e}^{\pi} \Theta))_0= \displaystyle{\psi + \{\pi, \gamma \}+  \frac{1}{2}  \{\pi, \{\pi, \mu \}\}+ \frac{1}{6}  \{\pi, \{\pi, \{\pi, \phi \}\}\}}\\
({\mathcal{M}}(\emph{e}^{\pi} \Theta))_1(P)= \displaystyle{\{\gamma, P \} + \{\{\pi, \mu \}, P \}+ \frac{1}{2} \{ \{\pi, \{\pi, \phi \}\},P \}}\\
({\mathcal{M}}(\emph{e}^{\pi} \Theta))_2(P,Q)=\bb{\bb{\mu}{P}}{Q}+ {\bb{\bb {\bb {\pi}{\phi}}{P}}{Q}}\\
({\mathcal{M}}(\emph{e}^{\pi} \Theta))_3(P,Q,R)={\bb{\bb {\bb {\phi}{P}}{Q}}{R}},
\end{array}
\right. $$
for all $P,Q,R \in \Gamma(\wedge^{\bullet} A)[2]$.
Now, the twisting of $l={\mathcal{M}}(\Theta)$ by $\pi$ is, according to (\ref{varepsilon_e_pi}) and (\ref{definition_l0})-(\ref{definition_l3}), given by $\varepsilon^\pi({\mathcal{M}}(\Theta))= \varepsilon^\pi  l=\sum_{i=0}^{3}(\varepsilon^\pi l)_i$ with
$$\left\{
  \begin{array}{ll}
(\varepsilon^\pi l)_0  = \displaystyle{\psi - \{ \gamma, \pi \}+  \frac{1}{2}  \{\{\mu, \pi\}, \pi \}- \frac{1}{6}  \{\{\{\phi, \pi\}, \pi\}, \pi \}}\\
(\varepsilon^\pi l)_1(P) =\displaystyle{ \{\gamma, P \} - \{\{\mu, \pi \}, P \}+ \frac{1}{2} \{\{ \{\phi, \pi\}, \pi \},P \}}\\
 (\varepsilon^\pi l)_2(P,Q)  =  \bb{\bb{\mu}{P}}{Q}- {\bb{\bb {\bb {\phi}{\pi}}{P}}{Q}}\\
 (\varepsilon^\pi l)_3(P,Q,R) = {\bb{\bb {\bb {\phi}{P}}{Q}}{R}},
\end{array}
\right.$$
for all $P,Q,R \in \Gamma(\wedge^{\bullet} A)[2]$.
Thus,
\begin{equation*}
(\varepsilon^\pi l)_i=({\mathcal{M}}(e^{\pi}\Theta))_i,
\end{equation*}
for $i=0,\cdots,3$, which concludes the proof.
\end{proof}

The next corollary is a consequence of the previous results.
\begin{cor}
The following assertions are equivalent:
  \begin{enumerate}
    \item $(A\oplus A^*, e^{\pi}\Theta)$ is a Courant algebroid;
    \item $(L,\varepsilon^\pi l)$ is a multiplicative curved $L_\infty$-algebra.
  \end{enumerate}
\end{cor}

Next, we show that the $L_\infty$-algebra attached to a bivector introduced in \cite{SchatzZambon17, SchatzZambon18} can be deduced from Corollary \ref{curved_L*_Courant} as a particular case. This way, one can  avoid the long direct proof presented in \cite{SchatzZambon17}.

Let $(A, \mu)$ be a Lie algebroid over $M$ and take a bivector $\pi \in \Gamma(\wedge^2 A)$.
 %Denote by $[\cdot, \cdot]_{_{SN}}$ the Schouten-Nijenhuis bracket on the space $\Gamma(\wedge^\bullet A)$ of multivectors of $A$.
 Next lemma appears in \cite{roy} for the case $A=TM$.

\begin{lem} \label{twist_mu_pi}
$(A, \mu)$ is a Lie algebroid if and only if $(A\oplus A^*, e^\pi \mu)$ is a Courant algebroid.
\end{lem}

\begin{proof}
Note that $(A, \mu)$ is a Lie algebroid if and only if $(A\oplus A^*, \mu)$ is a Courant algebroid. Applying Proposition \ref{exponential} with $\Theta= \mu$, we have that if $(A, \mu)$ is a Lie algebroid then $(A\oplus A^*, e^\pi \mu)$ is a Courant algebroid. Conversely,
the twisting of $\mu$ by $\pi$ is
\begin{equation*}
e^\pi \mu= \mu + \{\pi, \mu \} +\frac{1}{2}  \{\pi, \{\pi, \mu \}\}
\end{equation*}
and, by bidegree reasons, we have that
$$\{e^\pi \mu, e^\pi \mu \}=0 \, \Rightarrow \, \{\mu, \mu \}=0.$$
So, if $(A\oplus A^*, e^\pi \mu)$ is a Courant algebroid, then $(A, \mu)$ is a Lie algebroid.
\end{proof}

The twisting of $\mu$ by $\pi$ can be written as
\begin{equation*}
e^\pi \mu= \mu + \{\pi, \mu \} -\frac{1}{2} [\pi, \pi]_{_{SN}},
\end{equation*}
where $[\cdot, \cdot]_{_{SN}}$ is the Schouten-Nijenhuis bracket on the space $\Gamma(\wedge^\bullet A)$ of multivectors of $A$.  It is a well-known fact that the triple
 $(\mu, \{\pi, \mu \}, -\frac{1}{2} [\pi, \pi]_{_{SN}})$ is a quasi-Lie bialgebroid structure on $(A^*, A)$ \cite{roy}.

The (curved) $L_\infty$-algebra on $\Gamma(\wedge^{\bullet} A^*)[2]$ that corresponds to the Courant algebroid $(A^* \oplus A, \mu + \{\pi, \mu \} -\frac{1}{2} [\pi, \pi]_{_{SN}})$  is given by (\ref{lambda_i}), with
$\phi=0$, $\gamma= \{\pi, \mu \}$ and $\psi = -\frac{1}{2} [\pi, \pi]_{_{SN}}$. More precisely, and denoting by $[\cdot, \cdot]_\pi$ the bracket on the space $\Gamma(\wedge^{\bullet} A^*)$ of forms on $A$, that is usually called the Koszul bracket, (\ref{lambda_i}) gives:
$$\left\{\begin{array}{ll}
 \lambda_1(\alpha)=\bb{\mu}{\alpha}  \\
 \lambda_2(\alpha,\beta)=\bb{\bb{\{\pi, \mu \}}{\alpha}}{\beta}= (-1)^{|\alpha|}[\alpha, \beta]_{\pi}  \\
 \lambda_3(\alpha,\beta,\eta)=\displaystyle{\bb{\bb{\bb{-\frac{1}{2} [\pi, \pi]_{_{SN}}}\alpha}\beta}\eta },
\end{array}
\right.$$
where $|\alpha|$ denotes the degree of $\alpha$ on the Gerstenhaber algebra $(\Gamma(\wedge ^{\bullet}A^*)[1], \wedge, [\cdot, \cdot]_{\pi})$.
For $A=TM$, this is the $L_\infty$-algebra $(\Omega(M)[2], \lambda_1 +\lambda_2 +\lambda_3)$ introduced in \cite{SchatzZambon17}.

%%%%%%%%%%%%%%%%%%%%%%%%%%%%%%%%%%%%%%%%%%%%%%%%%%%%%%%%%%%%%%%%%%%%%%%%%%%%%%%%%%%%%%%%%%%%%%
%%%%%%%%%%%%%%%%%%%%%%%%%%%%%%%%%%%%%%%%%%%%%%%%%%%%%%%%%%%%%%%%%%%%%%%%%%%%%%%%%%%%%%%%%%%%%%%%%%%
\section{Twisting and deformation} \label{Twisting and deformation}

In this section we combine the operations of twisting and deformation on both (pre-)Courant algebroids and curved (pre-)$L_\infty$-algebras.

Let $\pi \in \Gamma(\wedge^2 A)$ be a bivector. Take $N \in \Gamma(A \otimes A^*)$   such that $N \circ \pi^\#= \pi^\# \circ N^*$ and consider the bivector $\pi_N \in \Gamma(\wedge^2 A)$ defined, for all $\alpha, \beta \in \Gamma (A^*)$, by
%\begin{equation*}
$\pi_N(\alpha, \beta)= \pi(N^* \alpha, \beta)$
%\end{equation*}
or, using the big bracket, by $\pi_N= \frac{1}{2} \{\pi, N \}$.
Mimicking the twisting of a pre-Courant structure by $\pi$, we may define the twisting of $N$ by $\pi$, and set
$e^{\pi}N := N+ \{\pi, N\}$.
We denote by $J_N$ and $J_{\pi N}$ the skew-symmetric endomorphisms of $A \oplus A^*$ given, respectively, by
$$J_N= \left(
\begin{array}{cc}
 N & 0 \\
0 & -N^*
\end{array}
\right)\,\,\,\,
{\textrm{and by}} \,\,\,\,
J_{\pi_{N}}= \left(
\begin{array}{cc}
 N & 2 \pi_N^\# \\
0 & -N^*
\end{array}
\right).
$$
 $J_N$ and $J_{\pi_{N}}$ are  identified with $N$ and $e^{\pi}N$, respectively, since
$$ J_N(X+\alpha)= \{X+\alpha, N\} \quad {\textrm{and}} \quad J_{\pi_{N}}(X+\alpha)= \{X+\alpha,e^{\pi}N \}.$$
The deformation by $J_N$ of the (pre-)Courant structure $\Theta= \psi + \gamma + \mu + \phi$ is the pre-Courant structure $\Theta_{N}=\{N , \Theta \}$, while the deformation by $J_{\pi_{N}}$ of the (pre-)Courant structure $e^\pi \Theta$ is the pre-Courant\footnote{Note that in general $\{N, \Theta \}$ is a pre-Courant structure, even if $\Theta$ is Courant. When $\Theta$ is Courant, $\{N, \Theta \}$ is Courant if and only if $N$ is a weak-Nijenhuis morphism \cite{grab}.} structure $(e^\pi \Theta)_{e^{\pi}N}=\{e^{\pi}N, e^\pi \Theta \}$.
 On the other hand, the twisting of $\Theta_{N}$ by $\pi$ is the pre-Courant structure $e^\pi (\Theta_N)$. The relation between these functions on $\mathcal{F}^{3}$ is given in the next proposition.

\begin{prop} \label{Theta_twist_J}
We have,
\begin{equation*}
e^\pi (\Theta_{N}) = (e^\pi \Theta)_{e^{\pi}N}.
\end{equation*}
Equivalently, the next diagram is commutative:

\begin{equation*}\label{diagram2}
\vcenter{
\xymatrixcolsep{6pc}\xymatrix@R=3.5pc{
    \Theta \ar@{->}[r]^{e^\pi}
     \ar@{->}[d]_{N}\ar@{}[dr]|{\rotatebox{270}{\scalebox{2}{$\circlearrowleft$}}}
     &  e^\pi \Theta
      \ar@{->}[d]^{e^\pi N}
      \\
    \Theta_{N}\ar@{->}[r]^(.4){e^\pi}&  (e^\pi \Theta)_{e^\pi N}
    }
    }
\end{equation*}

\end{prop}

\begin{proof}
We compute,

\begin{align*}
e^\pi (\Theta_{N}) & = e^\pi \{N, \Theta \}=  \{N, \mu \} + \{\pi, \{N, \phi \}\} \\
 & +  \{N, \gamma \} + \{\pi, \{N, \mu \}\} + \frac{1}{2}\{\pi, \{\pi, \{N, \phi \}\}\} +  \{N, \phi \} \\
 & +  \{N, \psi \} + \{\pi, \{N, \gamma \}\} +  \frac{1}{2}\{\pi, \{\pi, \{N, \mu \}\}\} +  \frac{1}{6} \{\pi, \{\pi, \{\pi, \{N, \phi \}\}\} \}.
\end{align*}
Applying the Jacobi identity of the big bracket, we get
\begin{align*}
e^\pi \{N, \Theta \} & = \{N, e^\pi \Theta \} + \{ \{ \pi, N \},  \mu + \{ \pi, \mu \} \} +  \{ \{ \pi, N \},  \gamma \}\\
& + \{ \{ \pi, N \},  \phi + \{ \pi, \phi \}  + \frac{1}{2} \{ \pi, \{ \pi, \phi \}\}\}.
\end{align*}
By bidegree reasons,
\begin{equation*}
\{ \{ \pi, N \},  \frac{1}{2}\{\pi, \{\pi,  \mu \}\}\}=0, \quad \{ \{ \pi, N \}, \{ \pi, \gamma \}\}=0, \quad  \{\pi,\{ \pi, \gamma \}\}=0
\end{equation*}
and
\begin{equation*}
\{ \{ \pi, N \},  \{ \pi, \{\pi, \{\pi,  \phi \} \}\}\}=0,
\end{equation*}
and so we may write
\begin{align*}
e^\pi \{N, \Theta \}  = & \{N, e^\pi \Theta \} + \{ \{ \pi, N \}, e^\pi \Theta \} \\
 = & (e^\pi \Theta)_{e^{\pi}N}.
\end{align*}
\end{proof}

\begin{cor} \label{Courant_Nijenhus_twist}
If $J_N$ is a Nijenhuis morphism on the Courant algebroid $(A \oplus A^*, \Theta)$, then
$e^\pi (\Theta_{N}) = (e^\pi \Theta)_{e^{\pi}N}$ is a Courant structure on $A \oplus A^*$.
%all the functions of $\mathcal{F}^{3}$ in diagram  (\ref{diagram2}) are Courant structures on $A \oplus A^*$.
\end{cor}

\begin{proof}
If $J_N$ is a Nijenhuis morphism on the Courant algebroid $(A \oplus A^*, \Theta)$, then $\Theta_{N}$ is a Courant structure on $A \oplus A^*$ \cite{A10}. By Proposition \ref{exponential}, $e^\pi (\Theta_{N})$ is a Courant structure on $A \oplus A^*$.
\end{proof}

Next we shall see how Proposition \ref{Theta_twist_J}  and Corollary \ref{Courant_Nijenhus_twist} translate into curved (pre-)$L_\infty$-algebras.

\begin{prop} \label{prop_twist_J_Lie}
Let $(A \oplus A^*, \Theta)$ be a pre-Courant algebroid, $(L,l=\mathcal{M}(\Theta))$ the curved pre-$L_\infty$-algebra determined by $\mathcal{M}$ and
$(\Gamma(\wedge^{\bullet} A)[2], \varepsilon^\pi l)$ its twisting by $\pi \in \Gamma(\wedge^2 A)$. Consider $\mathcal{j}'= {\mathcal{j}_1}+ {\underline{N}}({\mathcal{j}_0})$, with $\mathcal{j}_0$ and $\mathcal{j}_1$ given by (\ref{definition_j0}) and (\ref{definition_j1}). Then,
\begin{equation} \label{twist_J_Lie}
\varepsilon^\pi ( l_{\mathcal{j}_1}) = (\varepsilon^\pi l)_{\mathcal{j}'}.
\end{equation}
Equivalently, the next diagram is commutative:

\begin{equation*}\label{diagram3}
\vcenter{
\xymatrixcolsep{6pc}\xymatrix@R=3.5pc{
    l \ar@{->}[r]^{\varepsilon^\pi}
     \ar@{->}[d]_{\mathcal{j}_1}\ar@{}[dr]|{\rotatebox{270}{\scalebox{2}{$\circlearrowleft$}}}
     &  \varepsilon^\pi l
      \ar@{->}[d]^{\mathcal{j}'}
      \\
    l_{\mathcal{j}_1}=[\mathcal{j}_1,l]\ar@{->}[r]^(.45){\varepsilon^\pi}&[\mathcal{j}', \varepsilon^\pi l]=(\varepsilon^\pi l)_{\mathcal{j}'}
    }
    }
\end{equation*}

\end{prop}

\begin{proof}
We start by applying Proposition \ref{Theta_twist_J} to the pre-Courant structure $\Theta_{J_N}$, to get
\begin{equation} \label{proof 1}
{\mathcal{M}}(e^{\pi} (\Theta_{N}))= {\mathcal{M}}((e^{\pi} \Theta)_{e^{\pi}N}).
\end{equation}
Then, applying Theorem \ref{thm_diagrama_comutativo} on the right-hand side of (\ref{proof 1}), for the pre-Courant structure $e^{\pi} \Theta$ and the endomorphism $J_{\pi_{N}}$, gives
\begin{equation*}
{\mathcal{M}}((e^{\pi} \Theta)_{e^{\pi}N})=({\mathcal{M}}(e^{\pi} \Theta))_{\Upsilon(e^{\pi}N)}= ({\mathcal{M}}(e^{\pi} \Theta))_{\mathcal{j}'},
\end{equation*}
since by (\ref{definition_j0}) and (\ref{definition_j1}),
\begin{equation*}
\Upsilon(e^{\pi}N)= \Upsilon( \{\pi, N \}+ N)= - \{\pi, N \} + {\mathcal{j}_1}={\underline{N}}({\mathcal{j}_0})+ {\mathcal{j}_1}.
\end{equation*}
Using Proposition \ref{prop_M_exp_pi_commute}, we have
\begin{equation*}
({\mathcal{M}}(e^{\pi} \Theta))_{\mathcal{j}'}=(\varepsilon ^\pi l)_{\mathcal{j}'}.
\end{equation*}
Now, we take the left-hand side of (\ref{proof 1}) and apply Proposition \ref{prop_M_exp_pi_commute} to it, to get
\begin{equation*}
{\mathcal{M}}(e^{\pi} \Theta_{N})= \varepsilon ^\pi {\mathcal{M}}(\Theta_{N}).
\end{equation*}
Then, Theorem \ref{thm_diagrama_comutativo} for $\Theta$ and the endomorphism $J_{N}$, gives
\begin{equation*}
\varepsilon ^\pi {\mathcal{M}}(\Theta_{N})=\varepsilon ^\pi (l_{\Upsilon(N)})=\varepsilon ^\pi (l_{\mathcal{j}_1}),
\end{equation*}
which completes the proof.
\end{proof}

\begin{rem}
If $(A\oplus A^*, \Theta)$ is a Courant algebroid, then $\varepsilon^\pi \, l_{\mathcal{j}_1} = (\varepsilon^\pi l)_{\mathcal{j}'}$ is a curved $L_\infty$-algebra structure on $(\Gamma(\wedge^{\bullet} A)[2]$.
\end{rem}

\begin{cor}
If $\mathcal{j}_1 =\Upsilon(N)$ is a Nijenhuis form on the curved $L_\infty$-algebra $(\Gamma(\wedge^\bullet A)[2],l)$, then $l_{\mathcal{j}_1}$ and $\varepsilon^\pi(l_{\mathcal{j}_1})=(\varepsilon^\pi l)_{\mathcal{j}'}$ are curved $L_\infty$-algebra structures on $\Gamma(\wedge^{\bullet} A)[2]$.
\end{cor}

\begin{proof}
 If $\mathcal{j}_1=\Upsilon(N)$ is Nijenhuis for $l$, then $l_{\mathcal{j}_1}=[\mathcal{j}_1, l]=[\underline{N}, l]$ is a curved $L_\infty$-algebra structure on $\Gamma(\wedge^\bullet A)[2]$ \cite{ALC15}. By Proposition \ref{prop_twist_curved}, $\varepsilon^\pi(l_{\mathcal{j}_1})$ is a curved $L_\infty$-algebra on $\Gamma(\wedge^\bullet A)[2]$.
\end{proof}

The results of Theorem \ref{thm_diagrama_comutativo} and Propositions \ref{prop_M_exp_pi_commute}, \ref{Theta_twist_J} and \ref{prop_twist_J_Lie} can be combined to form the following commutative cubic diagram:

\begin{equation*}\label{cubo}
\vcenter{
\xymatrixcolsep{4pc}\xymatrix@R=2.5pc{
&\varepsilon^\pi l \ar@{->}[rr]^{\textstyle{\mathcal{j}'}}&&(\varepsilon^\pi l)_{\mathcal{j}'}
\\
l \ar@{->}[ur]^{\textstyle{\varepsilon^\pi}} \ar@{->}[rr]^(.6){\textstyle{\mathcal{j}_1}}&&l_{\mathcal{j}_1}\ar@{->}[ur]^{\textstyle{\varepsilon^\pi}}&
\\
&\emph{e}^{\pi}\Theta \ar@{->}'[r]^(.65){\textstyle{\emph{e}^{\pi}N}}[rr] \ar@{->}'[u]^{\textstyle{\mathcal{M}}} [uu]&&{(\emph{e}^{\pi}\Theta)}_{\emph{e}^{\pi}N}\ar@{->}[uu]_{\textstyle{\mathcal{M}}}
\\
\Theta\ar@{->}[ur]^{\textstyle{\emph{e}^{\pi}}}\ar@{->}[rr]^{\textstyle{N}}\ar@{->}[uu]^{\textstyle{\mathcal{M}}}&&\Theta_{N}\ar@{->}[ur]_{\textstyle{\emph{e}^{\pi}}}\ar@{->}[uu]^(.3){\textstyle{\mathcal{M}}}&
    }
    }
\end{equation*}

%%%%%%%%%%%%%%%%%%%%%%%%%%%%%%%%%%%%%%%%%%%%%%%%%%%%%%%%%%%%%%%%%%%%%%%%%%%%%%%%%%%%%%%%%%%%%%%%%%%%%%%%%%%%%%%%%%%%%%%%%%%%%%%%%%%%%%%%%%%%%%%%%%
%%%%%%%%%%%%%%%%%%%%%%%%%%%%%%%%%%%%%%%%%%%%%%%%%%%%%%%%%%%%%%%%%%%%%%%%%%%%%%%%%%%%%%%%%%%%%%%%%%%%%%%%%%%%%%%%%%%%%%%%%%%%%%%%%%%%%%%%%%%%%%%%%%%
%%%%%%%%%%%%%%%%%%%%%%%%%%%%%%%%%%%%%%%%%%%%%%%%%%%%%%%%%%%%%%%%%%%%%%%%%%%%%%%%%%%%%%%%%%%%%%%%%%%%%%%%%%%%%%%%%%%%%%%%%%%%%%%%%%%%%%%%%%%%%%%%%%%%%

\bigskip

\noindent {\bf Acknowledgments.} The authors would like to thank Alfonso Tortorella and James Stasheff for several comments and suggestions on a preliminary version of
this paper.  This work was partially supported by the Centre for Mathematics of the University of Coimbra - UIDB/00324/2020, funded by the Portuguese Government through FCT/MCTES.

%%%%%%%%%%%%%%%%%%%%%%%%%%%%%%%%%%%%%%%%%%%%%%%%%%%%%%%%%%%%%%%%%%%%%%%%%%%%%%%%%%%%%%%%%%%%%%%%%%%%%%%%%%%%%%%%%%%%%%%%%%%%%%%%%%%%%%%%%%%%%%%%%%

\end{document}